\documentclass{amsart}
\usepackage{amsfonts, bbm}

\usepackage{graphicx}
\usepackage{psfrag}

\newtheorem{theorem}{Theorem}[section]
\theoremstyle{plain}

\newtheorem{lemma}[theorem]{Lemma}
\newtheorem{proposition}[theorem]{Proposition}
\theoremstyle{remark}

\numberwithin{equation}{section}

\newcommand{\tr}{\operatorname{tr}}
\newcommand{\otr}{\operatorname{0-tr}}

\newcommand{\ovol}{\operatorname{0-vol}}
\newcommand{\FP}{\operatornamewithlimits{FP}}

\newcommand{\re}{\operatorname{Re}}
\newcommand{\im}{\operatorname{Im}}
\newcommand{\res}{\operatorname{Res}}
\newcommand{\rank}{\operatorname{rank}}

\newcommand{\supp}{\operatorname{supp}}
\newcommand{\detkv}{\operatorname{det_{\rm KV}}}
\newcommand{\tsp}{^{\rm t}}
\newcommand{\norm}[1]{\Vert #1 \Vert}
\newcommand{\bnorm}[1]{\left\Vert #1 \right\Vert}
\newcommand{\brak}[1]{\langle #1 \rangle}

\newcommand{\bbR}{\mathbb{R}}
\newcommand{\bbH}{\mathbb{H}}
\newcommand{\bbC}{\mathbb{C}}
\newcommand{\bbZ}{\mathbb{Z}}
\newcommand{\bbN}{\mathbb{N}}

\newcommand{\calR}{\mathcal{R}}

\newcommand{\calF}{\mathcal{F}}
\newcommand{\cinf}{C^\infty}
\newcommand{\del}{\partial}

\newcommand{\barX}{{\bar X}}
\newcommand{\bX}{{\partial\barX}}
\newcommand{\FPe}{\FP_{\varepsilon\to 0}}

\newcommand{\Ds}{(\Delta_g-s(n-s))}
\newcommand{\vep}{\varepsilon}
\newcommand{\tS}{\tilde{S}}
\newcommand{\tA}{\tilde{A}}
\newcommand{\hD}{\hat\Delta}
\newcommand{\hR}{\hat R}
\newcommand{\Drel}{D_{\rm rel}}
\newcommand{\Rsc}{\calR^{\rm sc}}
\newcommand{\Nsc}{N^{\rm sc}}
\newcommand{\nh}{\tfrac{n}{2}}

\begin{document}

\title[Upper and lower bounds]{Upper and lower bounds on resonances for manifolds hyperbolic near infinity}
\author[Borthwick]{David Borthwick}
\address{Department of Mathematics and Computer Science, Emory
University, Atlanta, Georgia, 30322, U. S. A.}
\email{davidb@mathcs.emory.edu}
\date{\today}

\begin{abstract}
For a conformally compact manifold that is hyperbolic near infinity and of dimension $n+1$, we 
complete the proof of the optimal $O(r^{n+1})$ upper bound on the resonance counting
function, correcting a mistake in the existing literature.  In the case of 
a compactly supported perturbation of a hyperbolic manifold,
we establish a Poisson formula expressing the regularized wave trace as a
sum over scattering resonances.  This leads to an $r^{n+1}$ lower bound on the
counting function for scattering poles.
\end{abstract}

\maketitle
\tableofcontents

\section{Introduction}\label{intro.sec}
If $(\barX, \bar{g})$ is a compact manifold with boundary and $\rho$ a boundary-defining function
for $\bX$, then the complete Riemannian manifold $X$ with metric $g := \rho^{-2} \bar{g}$ is 
called \emph{conformally compact}.  This definition is modeled on hyperbolic manifolds; 
for a discrete torsion-free group $\Gamma$ of isometries of $\bbH^{n+1}$, the 
quotient $\bbH^{n+1}/\Gamma$ is conformally compact precisely when 
$\Gamma$ is \emph{convex cocompact} (i.e.~the convex core of $\bbH^{n+1}/\Gamma$ is compact).  
In this paper we will be 
concerned with conformally compact manifolds $(X, g)$ which are \emph{hyperbolic
near infinity}, which means that $g$ has constant sectional curvature $-1$ outside of a 
compact set.  For any conformally compact manifold, the choice of boundary defining function 
$\rho$ induces a metric  $h = {\bar g}|_{\bX}$ on $\bX$, 
whose conformal class is defined independently of $\rho$.

For $(X, g)$ conformally compact and hyperbolic near infinity, we let $\dim X = n+1$
and denote by $\Delta_g$ the positive Laplacian associated to $g$.   The 
resolvent $R_g(s) := \Ds^{-1}$ has a meromorphic continuation to $s\in\bbC$
with poles of finite rank \cite{MM:1987, GZ:1995b}.   For background on the spectral theory 
of asymptotically hyperbolic manifolds, we refer the reader to Perry's survey article
\cite{Perry:2007}.

The \textit{resonances} of $(X,g)$ are the poles of $R_g(s)$ with multiplicities  
given by
$$
m_g(\zeta) := \rank \res_\zeta R_g(s).
$$
Resonances are closely related to the poles of the scattering matrix $S_g(s)$, which is
defined as in \cite{JS:2000, GZ:2003}.  For $\re s = \nh$, $s\ne \nh$, a function $f_1\in \cinf(\bX)$
determines a unique solution of $\Ds u = 0$ such that
$$
u \sim \rho^{n-s} f_1 + \rho^s f_2
$$
as $\rho \to 0$, with $f_2\in \cinf(\bX)$
This defines the map $S_g(s): f_1 \mapsto f_2$, which extends meromorphically to $s\in \bbC$ as
a family of pseudodifferential operators of order $2s-n$.
To define the multiplicity of scattering poles, we use a renormalized scattering matrix of order zero
given by
\begin{equation}\label{tS.def}
\tS_g(s) := \frac{\Gamma(s-\nh)}{\Gamma(\nh - s)}
\Lambda^{n/2-s} S_g(s) \Lambda^{n/2-s}.
\end{equation}
where
$$
\Lambda := \frac12 (\Delta_h+1)^{1/2}.
$$
This renormalization makes $\tS_g(s)$ into a meromorphic family of Fredholm operators with 
poles of finite rank.  The multiplicity at a pole or zero of $S_g(s)$ (with poles counted 
positively to match the resonances) is then defined by
$$
\nu_g(\zeta) := - \tr \bigl[\res_\zeta \tS_g'(s) \tS_g(s)^{-1}\bigr].
$$

The scattering multiplicities are related to the resonance multiplicities by
\begin{equation}\label{nu.mumu}
\nu_g(\zeta) = m_g(\zeta) - m_g(n-\zeta) + \sum_{k\in \bbN} \Bigl( \mathbbm{1}_{n/2 - k}(s) 
-\mathbbm{1}_{n/2+k}(s) \Bigr) d_k,
\end{equation}
where
$$
d_k := \dim \ker \tS_{g}(\tfrac{n}2+k).
$$
This result was partially established by Guillop\'e-Zworksi \cite{GZ:1997} (for $n=1$) and
Borthwick-Perry \cite{BP:2002} (for $\zeta \notin \frac{n}2 \pm \bbN$), 
and completed by Guillarmou \cite{Gui:2005} (with a restriction that was later removed 
in \cite{GN:2006}).  Guillarmou's computation of the correction term at half-integer points
was based on work of Graham-Zworski \cite{GZ:2003}, 
who identified $\tS_{g}(\tfrac{n}2+k)$ with a multiple of the $k$-th conformal 
Laplacian on $(\bX, h)$.  

We will distinguish two resonance sets, the set $\calR_g$ of resonances listed according to
multiplicities, and the \emph{scattering resonance} set
$$
\Rsc_g := \calR_g \cup \bigcup_{n=1}^\infty 
\Bigr\{\tfrac{n}2 - k \text{ with multiplicity } d_k\Bigr\}.
$$
The latter is not quite the same as the set of \textit{scattering poles}, which is
usually defined as the set where $\nu_g(\zeta) >0$. 
Note, however, that the multiplicities of points in $\Rsc_g$ differ from the scattering pole 
multiplicity $\nu_g(\zeta)$ only when $\zeta(n-\zeta) \in \sigma_{\rm d}(\Delta_g)$, 
i.e.~only at finitely many points.
We also introduce the respective counting functions,
\begin{equation}\label{NNsc}
N(r) :=  \#\{\zeta\in \calR_g:\> |\zeta| \le r\},\qquad
\Nsc(r) :=  \#\{\zeta\in \Rsc_g:\> |\zeta| \le r\},
\end{equation}
and note that $N(r) \le \Nsc(r)$.  

The difference between $N(r)$ and $\Nsc(r)$ can be significant.  For example, in 
$\bbH^{n+1}$ we can write the scattering matrix explicitly in terms of the Laplacian on
$S^n$, using \cite[Lemma~A.2]{GZ:1995a}, 
\begin{equation}\label{S0.formula}
\frac{\Gamma(s-\nh)}{\Gamma(\nh - s)}\>S_0(s) =  2^{n-2s} 
\frac{\Gamma(\nh - s)\> \Gamma\!\left(\sqrt{\Delta_{S^n} + (\frac{n-1}2)^2} + \tfrac12 + s - \nh\right)}{\Gamma(s-\nh)\>
\Gamma\!\left(\sqrt{\Delta_{S^n} + (\frac{n-1}2)^2} + \tfrac12 - s + \nh\right)}.
\end{equation}
From this we can quickly deduce that $\Rsc_0 = -\bbN_0$ with the multiplicity at $-k$ given by
\begin{equation}\label{nu0.def}
h_n(k) := (2k+n) \frac{(k+1)\dots (k+n-1)}{n!},
\end{equation}
which is the dimension of the space of spherical harmonics
of degree $k$ in dimension $n+1$.   Hence $\Nsc(r) \sim c_n k^{n+1}$.
If $n$ is odd, then the resonance set is $-\bbN_0$ with 
multiplicities given by $h_n(k)$, and
the two counting functions in (\ref{NNsc}) are the same.  However, for $n$ 
even $\bbH^{n+1}$ has no resonances, 
and in this case $\Nsc(s)$ is counting only the contributions from the $d_k$.

\begin{theorem}\label{upbound.thm}
For $(X,g)$ conformally compact and hyperbolic near infinity, 
\begin{equation}\label{Nsc.bound}
\Nsc(r) = O(r^{n+1}).
\end{equation}
\end{theorem}
In this context, Guillop\'e-Zworski \cite{GZ:1995b} proved the upper bound $N(r)  = O(r^{n+2})$,
along with the optimal $O(r^{2})$ bound for surfaces \cite{GZ:1995a}. 
Froese-Hislop \cite{FH:2000} sketched arguments for an $O(r^{n+1})$ bound in
the half-plane $\re s < 0$, under the assumption that the ends are asymptotic to
product metrics.  Cuevas-Vodev \cite{CV:2003} proved the $O(r^{n+1})$ bound in
a sector excluding the negative real axis, in the same context as Theorem~\ref{upbound.thm}.   
However, they did not establish the global bound (\ref{Nsc.bound}), as claimed.  The proof 
of \cite[Prop.~1.3]{CV:2003}, which covers the half-plane $\re s < 0$,
is flawed.\footnote{In \S2.2 they claim incorrectly that $\rho \circ \iota_j^{-1} = y + O(y^2)$.  The estimates 
in the appendix are consequently based on an oversimplified formula for the scattering matrix.}
One of the main contributions of this paper will be to prove the optimal estimate in the
half-plane $\re s < 0$ (see Proposition~\ref{left.upbound}), thus completing the 
proof of Theorem~\ref{upbound.thm}.

\bigbreak
Another primary result of this paper is a Poisson formula expressing the wave trace as 
a sum over the scattering resonance set.
This Poisson formula is stated in terms of the $0$-trace, a regularization introduced by
Guillop\'e-Zworski \cite{GZ:1997} for the surface case and inspired by the
b-integral of Melrose \cite{Melrose:1993}.   A conformally compact manifold is \emph{asymptotically
hyperbolic} if we can choose a boundary-defining function $\rho$ satisfying $|d\rho|_{\bar g} = 1$
on $\bX$.  In this case (which includes our setting) one can always choose a special defining function
such that $|d\rho|_{\bar g} = 1$ holds in some neighborhood of $\bX$ (see \cite{GL:1991}).  We will 
assume henceforth that $\rho$ satisfies this extra condition.

Suppose an operator $A$ has continuous kernel
$A(z,z')$, with respect to $dg$, and $A(z,z)$ admits a polyhomogeneous expansion
in $\rho$ as $\rho \to 0$.  Then we may define
$$
\otr A := \FPe \int_{\rho \ge \vep} A(z,z)\>dg(z),
$$
where FP denotes the finite part in the sense of Hadamard.
The $0$-volume of $(X,g)$ is similarly defined by 
$$
\ovol(X,g) := \FPe \text{vol}_g\{\rho \ge \vep\}.
$$
Assuming that $\rho$ is a special defining function as described above, 
this quantity is independent of the choice of $\rho$ in even dimensions, 
but not in odd dimensions.

Because the $0$-trace is purely formal, it is difficult to estimate directly.  Thus, in 
order to prove the Poisson formula we must introduce background
operators to reduce to actual traces.  In the two-dimensional context of \cite{GZ:1997},
the hyperbolic funnel and cusp boundary models give natural background operators.  
In higher dimensions, the only suitable candidate for the background is 
an exactly hyperbolic manifold with the same ends.  
Let $(X,g)$ be a conformally compact manifold.  For the results below we 
will assume that there exists a conformally compact hyperbolic manifold $(X_0,g_0)$
(possibly disconnected) such that $(X-K,g) \cong (X_0 - K_0,g_0)$ for some compact sets
$K \subset X$ and $K_0 \subset X_0$.  (Note: this restriction does not apply to 
Theorem~\ref{upbound.thm}.)

\begin{theorem}[Poisson formula]\label{poisson.thm}
Let $(X,g)$ be compactly supported perturbation of a conformally
compact hyperbolic manifold, in the sense described above.  
Then, in a distributional sense on $\bbR - \{0\}$, 
$$
\otr \left[ \cos \left(t \sqrt{\smash[b]{\Delta_g - n^2/4}}\,\right) \right] 
= \frac12 \sum_{\zeta\in \Rsc_g} e^{(\zeta-n/2)|t|} - A(X) \frac{\cosh t/2}{(2\sinh |t|/2)^{n+1}},
$$
where
$$
A(X) := \begin{cases}0 & n \text{ odd }(\dim X\>\text{is even}),  \\
\chi(X) & n \text{ even }(\dim X\>\text{is odd}).  \end{cases}
$$
\end{theorem}

\medskip\noindent
\emph{Remarks:}
\begin{enumerate}
\item   For surfaces with hyperbolic ends this Poisson formula was proven by 
Guillop\'e-Zworski \cite[Thm.~5.7]{GZ:1997}.  And for conformally compact hyperbolic manifolds it
was proven by Guillarmou-Naud \cite[Thm.~1.1]{GN:2006}, using the factorization of the
Selberg zeta function from Patterson-Perry \cite{PP:2001}. 
\item  The formula of Theorem~\ref{poisson.thm} can be extended through $t=0$ if both sides 
are multiplied by $t^{m}$ for $m$ sufficiently large.  In the case where $(X,g)$ is a non-topological 
perturbation (i.e. $X=X_0$), we can take $m = n+1$.   
\end{enumerate}

For general asymptotically hyperbolic manifolds, Joshi-S\'a Barreto \cite[Thm.~4.2]{JS:2001} 
showed that the singularities of the wave $0$-trace are contained in the set of periods of
closed geodesics.  One consequence of the Poisson formula is 
that $\Rsc_g$ determines this set of singularities, which we would expect to determine the 
periods of closed geodesics of $(X, g)$.  Another consequence is the following
lower bound:

\begin{theorem}\label{lowerbound.thm}
For $(X,g)$ a compactly supported perturbation of a conformally
compact hyperbolic manifold, we have
$$
\Nsc(r) \ge c  B(X,g) \> r^{n+1}.
$$
where
$$
B(X,g) = \begin{cases}|\ovol(X,g)| & n \text{ odd }(\dim X\>\text{is even})  \\
|\chi(X)| & n \text{ even }(\dim X\>\text{is odd})  \end{cases}
$$
\end{theorem}

\medskip\noindent
\emph{Remarks:}
\begin{enumerate}
\item 
The derivation of Theorem~\ref{lowerbound.thm} from Theorem~\ref{poisson.thm} follows
the arguments in Guillop\'e-Zworski \cite{GZ:1997}.  They established the optimal 
lower bound on $N(r)$ for general surfaces with hyperbolic ends.  (For $n=1$, we have
$d_k = 0$ for all $k$, so that $N(r) = \Nsc(r)$.)\footnote{The proof of this in \cite[Lemma~2.8]{GZ:1997} 
is incomplete;  see \cite[\S8.5]{Borthwick} for a corrected version.}  
The same methods were adapted by Perry \cite{Perry:2003}
to prove Theorem~\ref{lowerbound.thm} for conformally compact hyperbolic manifolds.
\item  Prior to Theorem~\ref{lowerbound.thm} there have been no existence results
for resonances in the general case of conformally compact manifolds hyperbolic near infinity.  In the broader 
\emph{asymptotically hyperbolic} class (conformally compact with $|d\rho|_{\bar g} = 1$ on $\bX$)
there are currently no general bounds and no existence results for resonances.
\item  If $h$ is conformally flat and $n \ge 2$, then $\tS_g(\nh+k) = \Delta_h^k$ \cite{GZ:2003}.  
Hence $d_k$ equals the number of connected components of $\bX$.  In this case
$\Nsc(r) = N(r) + O(r)$ so the lower bound applies to $N(r)$ as well.
\item  Guillarmou-Naud \cite[Prop.~2.2]{GN:2006} show that if $n>2$ and $h$ lies in the conformal class of a
metric $h_0$ with constant nonzero sectional curvature $\kappa$, then
$$
d_k = \ker \biggl(\prod_{j=1}^k \Bigl(\Delta_{h_0} + \kappa (\nh-j)(\nh+j-1)\Bigr) \biggr).
$$
In particular, when $\kappa \le 0$, the sequence of $d_k$'s is bounded and $\Nsc(r) = N(r) + O(r)$
so the lower bound extends to $N(r)$ as above.
\end{enumerate}

\bigbreak
This paper is organized as follows.  In \S\ref{param.sec} we recall the parametrix
construction for the resolvent from Guillop\'e-Zworski \cite{GZ:1995b}.  We
use this to derive formulas for the scattering matrix in \S\ref{scm.sec}, producing
a renormalized version of the scattering determinant.
Growth estimates on the various components of these formulas are obtained
in \S\ref{growth.sec}.  In \S\ref{optub.sec} we apply these estimates to bound the renormalized scattering
determinant in a half-plane, completing the proof of Theorem~\ref{upbound.thm}.  
Global estimates on the renormalized scattering determinant are derived in 
\S\ref{det.est.sec}.  In \S\ref{fact.sec} we analyze the relative
scattering determinant between two metrics which agree near infinity.  We also
compare our renormalized scattering determinant to the more intrinsic scattering
determinant introduced by Guillarmou \cite{Gui:2005b} in even dimensions.
The Poisson formula (Theorem~\ref{poisson.thm}) is proven in \S\ref{poisson.sec},
and then applied to derive Theorem~\ref{lowerbound.thm} in \S\ref{lbr.sec}.
Finally, in \S\ref{scphase.sec} we define a regularized scattering phase and
show that it satisfies Weyl-type asymptotics.

\vskip12pt\noindent
\textbf{Acknowledgment.}   Thanks to Colin Guillarmou for some 
helpful remarks and corrections.

\section{Parametrix construction}\label{param.sec}

Let $(X,g)$ be an $(n+1)$-dimensional manifold which is conformally compact manifold and
hyperbolic near infinity.  In this context, Guillop\'e-Zworski \cite[Lemma~3.1]{GZ:1995b} gave 
a refinement of the more general Mazzeo-Melrose parametrix construction \cite{MM:1987}.  
In particular, they produced meromorphic families of bounded operators,
$$
M_N(s): \rho^N L^2(X, dg) \to \rho^{-N} L^2(X, dg),
$$
and compact operators,
$$
K_N(s): \rho^N L^2(X, dg) \to \rho^N L^2(X, dg),
$$
for $N\in\bbN$, such that for $\re s > -N+\nh$
\begin{equation}\label{dm.ik}
\Ds M_N(s) = I - K_N(s).
\end{equation}
Both $M_N(s)$ and $K_N(s)$ have simple poles with finite rank residues at the points $s\in -\bbN$.
The meromorphic continuation of the resolvent follows by application of the analytic
Fredholm theorem to invert $I - K_N(s)$, which yields
\begin{equation}\label{r.mik}
R_g(s) = M_N(s)(I-K_N(s))^{-1}
\end{equation}
for $\re s > -N+\nh$.

Later we will need to refer to the explicit formulas for the parametrix and error terms, 
so we will review the construction from \cite{GZ:1995b}.  The assumption of hyperbolic near infinity
guarantees the existence of a collection of neighborhoods $Y_j\subset \barX$ such that
$\cup Y_j$ covers a neighborhood of $\bX$, with isometries
$$
\iota_j:Y_j \to \{z \in \bbH^{n+1}:\> |z|<1\},
$$
where $\bbH^{n+1}$ is the upper half-space $\bbR^n\times\bbR_+$ with the standard hyperbolic metric.
On each $Y_j$ the isometry $\iota_j$ defines a set of coordinates $(x,y)\subset \bbR^n\times\bbR_+$.
We will set $U_j := Y_j\cap \bX$, so that $\{U_j\}$ forms an open cover
for $\bX$.  For each neighborhood we define smooth functions $\gamma_j \in \cinf(U_j)$ by
$$
\gamma_j = \lim_{\rho \to 0} \frac{y|_{Y_j}}{\rho}
$$

The parametrix is built from pullbacks of the model resolvent $R_0(s) := (\Delta_{\bbH^{n+1}} - s(n-s))^{-1}$.  
To patch the pieces together, we introduce a set of smooth functions $\chi^j$, supported in $Y_j$,
such that $\chi := \sum_j \chi^j$ is equal to $1$ in some neighborhood of $\bX$.
Moreover, the $\chi^j$ can be constructed in the form $\varphi^j \psi^j$, where 
$\psi_j$ depends only on $y$ in the $Y_j$ coordinates,  
$\{\varphi^j\}$ gives a partition of unity for $\bX$, and each $\varphi^j$ is extended into $Y_j$ as a function
that depends only on $x$.   We also introduce $\psi_1^j$ and $\varphi_1^j$, with strictly greater supports,
such that
$$
\psi_1^j \psi^j = \psi^j,\qquad \varphi_1^j\varphi^j = \varphi^j.
$$
Let $\chi_1^j = \psi_1^j \varphi_1^j$ and $\chi_1 := \sum_j \chi_1^j$.  Finally, let 
$\chi_0 \in \cinf(X)$ equal $1$ in some neighborhood of $\bX$, with 
support contained inside that of $\chi$ so that $\chi_0 \chi = \chi_0$. 

Choose $s_0$ with $\re s_0 > \nh$ such that $R_g(s_0)$ is well defined. 
The first step towards the parametrix is
$$
M_0(s) := (1-\chi_0) R_g(s_0) (1- \chi) +  \sum_j \chi_1^j \> \iota_j^*(R_0(s)) \chi^j.
$$
This gives
\begin{equation}\label{ds.m0}
\begin{split}
\Ds M_0(s) & = I - [\Delta_g, \chi_0] R_g(s_0) (1-\chi) \\
&\quad + (s_0(n-s_0) - s(n-s)) (1-\chi_0) R_g(s_0) (1-\chi) \\
&\quad + \sum_j \bigl[\Delta_g, \chi_1^j\bigr] \> \iota_j^*(R_0(s))\>  \chi^j.
\end{split}
\end{equation}

The construction proceeds by solving away error terms at the boundary.  For this purpose,
the substitution $u = y^2$ is used to alter the smooth structure in local coordinates.
With $w = (x,u)$, $w' = (x',u')$, the model resolvent kernel on $\bbH^{n+1}$ has the expansion
$$
R_0(s;z,z') = \sum_{k=0}^\infty  a_k(s) (uu')^{s/2+k} q(w,w')^{-s-2k},
$$
where
$$
q(w,w') := |x-x'|^2 + u + {u'},
$$ 
and
$$
a_k(s) := \frac{\pi^{-n/2} \Gamma(s+2k)}{2k!\> \Gamma(s - \tfrac{n}2+k+1)}.
$$
In the $w$-coordinates, we have
$$
\Ds u^{s/2+k} f = -4k(s-\tfrac{n}2+k) u^{s/2+k} f + u^{s/2+k+1} Q(\tfrac{s}2+k) f,
$$
with
$$
Q(t) := 2(n-2-4t)\del_{u} - 4u \del_{u}^2 + \Delta_x.
$$
The formula for the parametrix $M_N(s)$ is
$$
M_N(s) = M_0(s) + \sum_{p=0}^{N-1} \sum_j \psi_1^j N^j_{p}(s) \chi^j,
$$
where the integral kernel of $N^j_{p}(s)$ is given in the local coordinates for $Y_j$ by
$$
N^j_{p}(s;w,w') :=  u^{s/2+p+1} \sum_{k=0}^p b_{k,p}(s) \prod_{l=1}^{p-k} Q(\tfrac{s}2+k+l) 
[\Delta_x, \varphi_1^j] q(w,w')^{-s-2k} {u'}^{s/2+k}
$$
with $\prod_{l=1}^{p-k} Q(\tfrac{s}2+k+l)$ replaced by 1 if $p=k$, and
$$
b_{k,p}(s) := \frac{\pi^{-n/2} 2^{-2p+2k-3} \Gamma(s+2k)}{(p+1)!\> \Gamma(s - \frac{n}2 + p + 2)}
$$

After plugging the expression for $M_N(s)$ into (\ref{dm.ik}), we obtain the error term
\begin{eqnarray}
K_N(s) & = &  - [\Delta_g, \chi_0] R_g(s_0) (1-\chi) + (s_0(n-s_0) - s(n-s)) (1-\chi_0) R_g(s_0) (1-\chi) \notag\\
& & + \sum_j [\Delta_g, \psi_1^j]\biggl( \varphi_1^j R_0(s) + \sum_{p=0}^{N-1}
N_p^j(s)\biggr) \chi^j  \label{kn.formula}\\
& & + \sum_j  \psi_1^j (L_N^j(s) + L_\sharp^j(s)) \chi^j, \notag
\end{eqnarray}
where, in the coordinates of $Y_j$,
$$
L_N^j(s; w,w') :=  u^{s/2+N+1} \sum_{k=0}^{N-1} b_{k,N-1}(s) \prod_{l=1}^{N-k} Q(\tfrac{s}2+k+l) 
[\Delta_x, \varphi_1^j] q(w,w')^{-s-2k} {u'}^{s/2+k},
$$
and
$$
L_\sharp^j(s;w,w') :=  \psi_1^j \sum_{k=N}^{\infty} a_{k}(s) u^{s/2+k+1} 
[\Delta_x, \varphi_1^j] q(w,w')^{-s-2k} {u'}^{s/2+k} .
$$

The remainder term $K_N(s)$ is the sum of a compactly supported pseudodifferential
operator of order $-1$ and a smoothing term
with kernel contained in $\rho^{s+2N+2} {\rho'}^s \cinf(X\times X)$.  In
particular, $K_N(s)$ is compact on $\rho^N L^2(X, dg)$ for $\re s > -N+\nh$ and 
the formula (\ref{r.mik}) is valid in this range (assuming that $s_0$, which we have suppressed
from the notation, was chosen appropriately).
Furthermore, the operator $K_N(s)^{n+1}$ is trace class on $\rho^N L^2(X, dg)$, and 
\begin{equation}\label{DN.def}
D_N(s) := \det (I - K_N(s)^{n+1})
\end{equation}
defines a meromorphic function for $\re s > -N+\nh$.

A few extra assumptions are needed in order to produce estimates.  First of all,
we assume that for $\delta>0$, in each $Y_j$ coordinate system we have $y<\delta$ in the support
of $\psi^j$ and $y>2\delta$ in the support of $1-\psi_1^j$.  By changing the definition of 
$\rho$, if necessary, we may assume also that $\rho = 1$ when $y>\delta$ as well.
According to \cite[Thm.~1.4.2]{Hormander:I}, we can require that the derivatives
of $\varphi^j$ and $\varphi_1^j$ satisfy quasi-analytic bounds of the form
\begin{equation}\label{qanal}
\norm{D^\alpha \varphi}_\infty \le C^{|\alpha|} e^{|\alpha| \log |\alpha|},
\end{equation}
for any multi-index $\alpha$.  (For $\varphi_1^j$ this is equivalent to 
\cite[eq.~(4.1)]{GZ:1995b}, although stated slightly differently.)  Finally,
we can assume that $\rho$ is given by $\sum_j \varphi^j \>y|_{Y_j}$ near $\bX$, so that
the functions $\gamma_j$ satisfy estimates of the form (\ref{qanal}) also.

With these assumptions (and assuming $\delta$ sufficiently small) Guillop\'e-Zworski proved the following:
\begin{proposition}\label{GZprop}
Given $\eta>0$, there exists a constant $C$ independent of $N$ such that
\begin{equation}\label{KN.bound}
\norm{K_N(s)} \le e^{CN},
\end{equation}
and
\begin{equation}\label{DN.bound}
|D_N(s)| \le \det (I + |K_N(s)|^{n+1})\le e^{CN^{n+2}},
\end{equation}
for $N$ sufficiently large, $|s| < N/C$ and $d(s,-\bbN_0)>\eta$.
\end{proposition}
These results are paraphrased from \cite[eq.~(3.6), Prop.~4.1, and Lemma~5.2]{GZ:1995b}.

\section{Scattering matrix}\label{scm.sec}

The resolvent kernel gives rise to a generalized Poisson kernel defined by
$$
E_g(s;z,x') := \lim_{\rho'\to 0} {\rho'}^{-s} R_g(s;z,z'),
$$
for $z\in X$ and $x' \in \bX$, where we use $\rho'$ to denote $\rho(z')$ and
make the implicit assumption that $z' \to x'$ as $\rho \to 0$.
From the parametrix construction (\ref{r.mik}), it is not difficult to see that 
$E_g(s; \cdot,\cdot) \in \cinf(X\times \bX)$.  This kernel defines a \emph{Poisson operator},
$$
E_g(s): L^2(\bX, dh) \to  \rho^{-N} L^2(X,dg)
$$
for $\re s > -N+\nh$, by
$$
E_g(s)f := \int_\bX E_g(s;\cdot,x') f(x') \>dh(x').
$$
(Recall that $h$ is the metric on $\bX$ induced by $\rho^2 g$.)

The term Poisson operator refers to the fact that for $f_1 \in \cinf(\bX)$, $u = E_g(s)f_1$ 
is a solution of $\Ds u = 0$.  This solution is contained in $\rho^{n-s} \cinf(\barX) 
+ \rho^s \cinf(\barX)$, with leading behavior
\begin{equation}\label{Egf.lead}
(2s-n) E_g(s)f_1 \sim \rho^{n-s} f_1 + \rho^s f_2,
\end{equation}
for some $f_2 \in \cinf(\bX)$.  
The scattering matrix is defined as the map $S_g(s):f_1 \mapsto f_2$, which is a pseudodifferential
operator of order $2s-n$.  For details of the definitions of Poisson operator and scattering
matrix, see \cite{JS:2000, GZ:2003}.  

By the symmetry of the resolvent, $S_g(s) = S_g(s)\tsp$.  
For $\re s \ge \nh$, $s\notin \bbN/2$, solutions of $\Ds u = 0$ are uniquely specified by the 
$\rho^{n-s} f_1$ term in the boundary expansion.
Thus, by (\ref{Egf.lead}) we have $(2s-n) E_g(s) f_1 = - (2s-n) E_g(n-s) f_2$ for $\re s = \nh$, $s\ne \nh$.  
This implies some useful meromorphic identities:
\begin{equation}\label{SE.identities}
\begin{split}
S_g(s)^{-1} & = S_g(n-s),  \\
S_g(n-s) E_g(s)\tsp & = -E_g(n-s)\tsp.
\end{split}
\end{equation}
The off-diagonal integral kernel (with respect to $dh$) of the scattering matrix 
can be derived directly from the resolvent:
$$
S_g(s;x,x') = \lim_{\rho, \rho' \to 0}  (\rho\rho')^{-s} R_g(s;z,z')
$$
for $x\ne x'$.  This relationship is useful for extracting formulas for the scattering matrix kernel
from the parametrix construction for the resolvent.

By (\ref{r.mik}) we can write
\begin{equation}\label{rx.chi}
R_g(s) = M_N(s) + R_g(s) K_N(s),
\end{equation}
for $\re s > -N + \nh$.
Multiplying the kernels in this formula by $\rho^{-s}$ on the left and ${\rho'}^{-s}$
on the right
and taking the restriction to $\bX\times\bX$, off the diagonal, yields a formula for 
the scattering matrix.  The only contribution from $M_N(s)$ term is the operator 
\begin{equation}\label{A.def}
A(s) = \sum_j \varphi_1^j \gamma_j^s\> \iota_j^*(S_0(s)) \>\gamma_j^s \varphi^j,
\end{equation}
coming from the $M_0(s)$ term.
To denote the boundary limit of $K_N(s)$ we introduce
$$
B_N(s;z,x') := \lim_{\rho'\to 0} {\rho'}^{-s} K_N(s; z,z'),
$$
This kernel is contained in $\rho^{s+2N+2} \cinf(\barX\times \bX)$ and defines
a smoothing operator that maps $L^2(\bX, dh) \to \rho^N L^2(X)$ for $\re s > -N + \nh$.
With these definitions, (\ref{rx.chi}) gives
\begin{equation}\label{S.AEB}
S_g(s) = A(s) + E_g(s)\tsp B_N(s),
\end{equation}
for $\re s > -N + \nh$.

By the identities (\ref{SE.identities}) we can rewrite (\ref{S.AEB}) as
\begin{equation}\label{S.SAKB}
S_g(n-s)A(s) = I + E_g(n-s)\tsp B_N(s),
\end{equation}
which shows in particular that the Fredholm determinant of $S_g(n-s)A(s)$ is well-defined
(as a meromorphic function), since $E_g(n-s)\tsp B_N(s)$ is a smoothing operator.
We can thus define a renormalized scattering determinant by
\begin{equation}\label{vartheta.def}
\vartheta_g(s) := \det S_g(n-s)A(s).
\end{equation}

There are two variants of (\ref{S.SAKB}) that we will use to produce estimates of $\vartheta_g(s)$.
The first comes from using (\ref{r.mik}) to write
\begin{equation}\label{e.fik}
E_g(s)\tsp = F(s) (I - K_N(s))^{-1},
\end{equation}
for $\re s > -N + \nh$, where
\begin{equation}\label{F.def}
F(s; x,z') := \lim_{\rho \to 0} \rho^{-s} M_N(s; z,z').
\end{equation}
(This limit is independent of $N$ because only the $M_0(s)$ term contributes.)
Applying (\ref{e.fik}) in (\ref{S.SAKB}) gives
\begin{equation}\label{sa.fkb}
S_g(n-s)A(s) =  I + F(n-s) (I - K_N(n-s))^{-1} B_N(s),
\end{equation}
valid for $|\re s - \nh| < N$.

The second variant comes from using the transpose of (\ref{rx.chi}) to derive
$$
E_g(s)\tsp =  G_N(s)\tsp + B_N(s)\tsp R_g(s),
$$
for $\re s > - N + \nh$, where 
$$
G_N(s;z,x') := \lim_{\rho'\to 0} {\rho'}^{-s} M_N(s;z,z').
$$
In conjunction with (\ref{S.SAKB}), this gives
\begin{equation}\label{sa.gbrb}
S_g(n-s)A(s) = I + \Bigl(G_N(n-s)\tsp + B_N(n-s)\tsp R_g(n-s)\Bigr) B_N(s),
\end{equation}
valid for $|\re s-\nh| < N$.

\section{Growth estimates}\label{growth.sec}

In this section we will give estimates for the various operators appearing in
(\ref{sa.fkb}) and (\ref{sa.gbrb}).  Many of these are quite similar to the estimates
by Guillop\'e-Zworski \cite{GZ:1995b}.  We will control the growth of $\vartheta_g(s)$
by estimating the singular values of $S_g(n-s)A(s)$.  This reduces to a combination of singular value
estimates of the smoothing term $B_N(s)$ and operator-norm estimates of the other terms.
Throughout the section we will follow the convention that $C$ is a large constant whose
value may change from line to line.

From (\ref{kn.formula}) we can read off an expression for $B_N(s)$,
$$
B_N(s) = \sum_j B_N^j(s), 
$$
where the kernel of $B_N^j(s)$ is supported in $Y_j \times U_j$ and is
given in local coordinates by
\[
\begin{split}
B^j_N(s;w,x') & :=  [\Delta_g, \psi_1^j] \varphi_1^j E_0(s;w,x') \gamma_j^s \varphi^j \\
&\quad +  [\Delta_g, \psi_1^j] \sum_{p=0}^{N-1} 
u^{s/2+p+1} b_{0,p}(s) \prod_{l=1}^{p} Q(\tfrac{s}2+l) [\Delta_x, \varphi_1^j] q_1(w,x')^{-s}  \gamma_j^s \varphi^j \\
&\quad + \psi_1^j u^{s/2+N+1} b_{0,N-1}(s) \prod_{l=1}^{N} Q(\tfrac{s}2+l) 
[\Delta_x, \varphi_1^j] q_1(w,x')^{-s}  \gamma_j^s \varphi^j,
\end{split}
\]
where $w = (x,u)$ and $q_1(w,x') := |x-x'|^2 + u$.
\begin{lemma}\label{B.est}
Given $\eta>0$, there exist constants $C, c$ independent of $N$ such that 
$\rho^{-N} B_N(s)$, as a map $L^2(\bX, dh) \to L^2(X, dg)$, satisfies
$$
\norm{\rho^{-N} B_N(s)} \le e^{CN},
$$
$$
\mu_k(\rho^{-N} B_N(s)) \le e^{CN - ck^{1/n}}.
$$
for $N$ sufficiently large, $|s| < N/C$, and $d(s,-\bbN_0)>\eta$.
\end{lemma}
\begin{proof}
First of all, Stirling's formula can be used to deduce that
$$
b_{0,p}(s) \le e^{CN} p^{-2p},
$$
for $|s| \le N/C$, $d(s,-\bbN_0)>\eta$.  The fact that $q_1(w,x')^{-s}$ extends to
a holomorphic function on a neighborhood of 
$$
\Bigl\{x \in \supp \nabla \varphi_1^j,\> x' \in \supp \varphi^j, \>0<u<\delta^2\Bigr\}
$$
can be used to deduce
$$
\Bigl| \psi_1^j(\sqrt{u}) [\Delta_x, \varphi_1^j(x)] \varphi^j(x') D_w^\alpha D_{x'}^\beta q(w,x')^{-s} \Bigr|
\le C^{|\alpha|+|\beta|} (|\alpha| + |\beta|)! e^{CN},
$$
via Cauchy's estimate.  Finally, we have the quasi-analytic estimates (\ref{qanal}) which
apply to $\varphi_1^j$, $\varphi^j$, and $\gamma_j$.

Combining these ingredients exactly as in the proofs of 
\cite[Lemma~4.1, Lemma~4.2 and Prop.~4.1]{GZ:1995b}, we can deduce that
\begin{equation}\label{dm.bnj}
\bnorm{\rho^{-N} \Delta_{x'}^m B_N^j(s;w,x')}_\infty 
\le  C^{2m} (m+N)^{2m} e^{CN},
\end{equation}
for $m, N \in \bbN$, $|s| \le N/C$, $d(s,-\bbN_0)>\eta$.  Note that we place no restriction on
the number of derivatives $m$.  

Using these estimates for the components of $B_N(s)$, together with the fact that
the metric $h$ is related to the Euclidean metric in local coordinates $x'$ by powers
of $\gamma_j$, we can deduce estimates in the operator norm for $L^2(\bX, dh) \to
L^2(X, dg)$,
\begin{equation}\label{dm.bn}
\bnorm{\rho^{-N} B_N(s) \Delta_h^m} \le  C^{2m} (m+N)^{2m} e^{CN},
\end{equation}
By Weyl's asymptotic for the eigenvalues of $\Delta_h$, we have
$$
\mu_k((\Delta_h+1)^{-m}) \sim Ck^{-2m/n}.
$$
Combining this with (\ref{dm.bn}) gives the estimates
$$
\mu_k(\rho^{-N} B_N(s)) \le k^{-2m/n} C^{2m} (m+N)^{2m} e^{CN},
$$
for all $m\in \bbN$.

The final step is to optimize the choice of $m$.  For $k > (eCN)^{n}$, we set 
$$
m = \bigl[(eC)^{-1} k^{1/n} - N + 1 \bigr],
$$
and with this choice we have
$$
k^{-2m/n} C^{2m} (m+N)^{2m} \le e^{c_1N - c_2k^{1/n}}.
$$
\end{proof}

\bigbreak
The operator $F(s)$ defined in (\ref{F.def}) is given explicitly by
\begin{equation}\label{F.E0}
F(s) = \sum_j \varphi_1^j \gamma_j^s \> \iota_j^*(E_0(s)\tsp) \> \chi^j
\end{equation}
where $E_0(s)$is the Poisson operator on $\bbH^{n+1}$, with kernel
$$
E_0(s;z,x') :=  c(s)  \left(\frac{y}{|x-x'|^2+{y}^2}\right)^s,
$$
where 
$$
c(s) := 2^{-1} \pi^{-n/2} \frac{\Gamma(s)}{\Gamma(s-\nh+1)}.
$$
Because of the singularity at $x=x', y=0$, it's easiest to use the
Hilbert-Schmidt norm to estimate this expression.
\begin{lemma}\label{F.est}
Let $\norm{\cdot}_2$ denote the Hilbert-Schmidt norm for $L^2(X, dg) \to L^2(\bX, dh)$.
Given $\eta>0$, there exists a constant $C$ independent of $N$ such that
$$
\norm{F(s) \rho^N}_2 \le e^{CN},
$$
for $N$ sufficiently large, $|s| < N/C$ and $d(s,-\bbN_0)>\eta$.
\end{lemma}
\begin{proof}
It suffices to do the estimate on the individual local coordinate expressions in (\ref{F.E0}).  
Suppose that $\chi \in \cinf_0(\bbH^{n+1})$, $\varphi \in \cinf_0(\bbR^n)$ and $\gamma \in \cinf(\bbR^n)$, 
$\gamma >0$.  Using the Hilbert-Schmidt norm for operators $L^2(\bbR^{n}) \to L^2(\bbH^{n+1})$,
we have
$$
\norm{\chi y^{N} E_0(s) \gamma^s \varphi}_2^2 :=  |c(s)|^2 \int_{\bbR^n} \int_{\bbH^{n+1}}  
\frac{y^{2\re s + 2N}}{(|x-x'|^2+{y}^2)^{2\re s}} \chi(z)^2 \gamma(x')^{2 \re s}
\varphi(x')^2 \>\frac{dx\>dy}{{y}^{n+1}} \>dx',
$$
After introducing polar coordinates $r = \sqrt{|x-x'|^2+{y}^2}$ and $\omega = (x-x',y)/r$, we can bound
the integral by 
$$
C \Bigl(\sup_{\supp \phi} |\gamma|\Bigr)^{|2\re s|}
 \int_0^{c} \int_{S^n_+} r^{-2\re s + 2N-1} (\omega_{n+1})^{2\re s+2N - n-1}\>dr\>d\omega,
$$
where $c$ is determined by the supports of $\chi$ and $\varphi$.  
For $|s| < N/C$, assuming $C >2$ and $N>n$, we have $2N - |2\re s| > n$, so this expression is easily
bounded by $e^{CN}$.

To finish the proof, we use Stirling's formula and
$\Gamma(s) = \pi/(\Gamma(1-s) \sin \pi s)$ to produce a bound 
$$
|c(s)| \le C\brak{s}^{n/2-1},
$$
valid for $d(s, -\bbN_0) > \eta$.
\end{proof}

The operator $G_N(s)$ appearing in (\ref{sa.gbrb}) differs from $F(s)\tsp$ by a 
smoothing term whose kernel is given in local coordinates by
\begin{equation}\label{g.f}
\sum_{p=0}^{N-1} \sum_j \psi_1^j u^{s/2+p+1} b_{0,p}(s) \prod_{l=1}^{p} Q(\tfrac{s}2+l) 
[\Delta_x, \varphi_1^j] q(w,w')^{-s} \varphi^j.
\end{equation}
By \cite[Lemma~4.1]{GZ:1995b} the sup norm of this smooth kernel is bounded
by $e^{-N/C}$ for $|s| < N/C$, $d(s, -\bbN_0) > \eta$.  Combining this estimate with
Lemma~\ref{F.est} gives the following:
\begin{lemma}\label{G.est}
Given $\eta>0$, there exists a constant $C$ independent of $N$ such that,
using the operator norm for maps $L^2(X, dg) \to L^2(\bX, dh)$, we have
$$
\norm{G_N(s)\tsp \rho^N} \le e^{CN}.
$$
for $N$ sufficiently large, $|s| < N/C$, and $d(s,-\bbN_0)>\eta$.
\end{lemma}

\bigbreak
The final estimate is to use results of \cite{GZ:1995b} to control $(I - K_N(s))^{-1}$.
Let $U_m$ denote the set of $m$-th roots of unity, and define the canonical product
\begin{equation}\label{gn.def}
g_n(s) := s \prod_{k=1}^{\infty} \prod_{\omega\in U_{2(n+1)}} E\Bigl(-\frac{\omega s}{k}, n+1\Bigr)^{k^{n}},
\end{equation}
using the elementary factor,
\begin{equation}\label{elemfact}
E(z, p) := (1-z)\exp\Bigl(z+\frac{z^2}2+\dots+\frac{z^p}p\Bigr).
\end{equation}
The inclusion of the roots of unity guarantees, by
Lindel\"of's theorem \cite[Thm.~2.10.1]{Boas}, that $g_n(s)$ is of finite type, so that
\begin{equation}\label{gs.est}
|g_n(s)| \le e^{C|s|^{n+1}}.
\end{equation}

\begin{lemma}\label{IK.est}
For each $N$ there exists $a_N$ contained in some fixed interval $[a,b]$, such that
$$
\left\Vert (I-K_N(s))^{-1} \right\Vert \le e^{CN^{n+2+\vep}} \quad\text{for }|s| = a_N N,
$$
where the norm is the operator norm on $\rho^{N} L^2(X, dg)$, 
\end{lemma}
\begin{proof}
First we expand
$$
(I-K_N(s))^{-1} = (I+K_N(s) + \dots + K_N(s)^n) (I - K_N(s)^{n+1})^{-1}.
$$
The first factor on the right satisfies a bound
$$
\norm{I+K_N(s) + \dots + K_N(s)^n} \le e^{CN},
$$
for $|s| < N/C$, $d(s,-\bbN)>\eta$, by (\ref{KN.bound}).
By \cite[Thm.~5.1]{GK:1969}, we can estimate the second factor by a ratio
of determinants,
$$
\left\Vert (I-K_N(s)^{n+1})^{-1} \right\Vert \le \frac{\det (I + |K_N(s)|^{n+1})}{D_N(s)}.
$$
Estimation of the numerator is already taken care of by (\ref{DN.bound}).  

By \cite[Lemma~5.3]{GZ:1995b}, for some fixed $p$ (independent of $N$) the function
$$
h_N(s) := g_{n+1}(s)^p D_N(s)
$$ 
is holomorphic for $\re s > -N + \nh$ and satisfies
$$
|h_N(s)| \le e ^{CN^{n+2}},
$$
for $|s| < N/C$.  

By the minimum modulus theorem (see e.g.~\cite[Thm.~I.11]{Levin}), there exists $a_N \in 
[\tfrac{1}{6C}, \tfrac{1}{2eC}]$ such that 
$$
|h_N(s)| \ge e^{-CN^{n+2}}, \quad\text{for }|s| = a_N N.
$$
By (\ref{gn.def}) the same estimate applies to $D_N(s)$.  
\end{proof}

\section{Optimal upper bound}\label{optub.sec}

Our first application of the estimates from \S\ref{growth.sec} will be to complete
the proof of Theorem~\ref{upbound.thm}.  Since the result of Cuevas-Vodev \cite[Prop.~1.2]{CV:2003}
covers a sector away from the negative real axis, as illustrated
in Figure~\ref{rescount}, it suffices to prove the following:
\begin{proposition}\label{left.upbound}
For $X, g$ conformally compact and hyperbolic near infinity and $\vep>0$,
$$
\#\Bigl\{\zeta\in \Rsc_g:\>|\zeta|\le r, \> \arg(\zeta) \in [\tfrac{\pi}2 + \vep,\tfrac{3\pi}2 - \vep]  \Bigr\} = O(r^{n+1}).
$$
\end{proposition}

\begin{figure}[ht]
\psfrag{n}{$\nh$}
\begin{center}  
\includegraphics{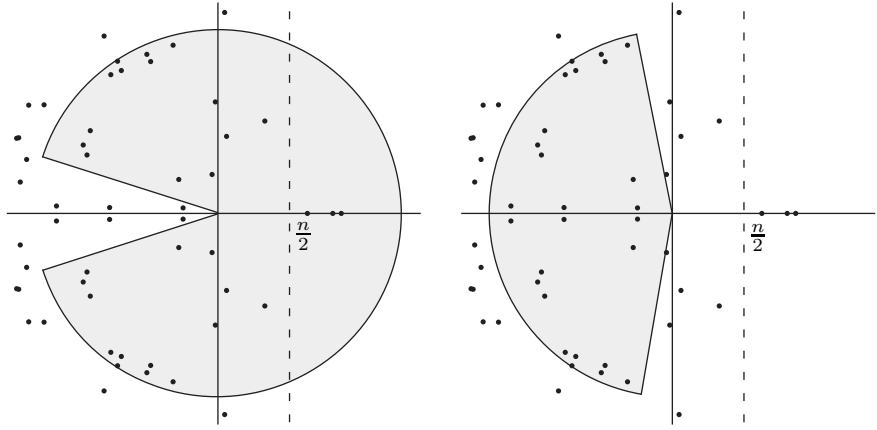} 
\end{center}
\caption{Resonance counting regions for the Cuevas-Vodev result (left) and 
for Proposition~\ref{left.upbound} (right).} \label{rescount}
\end{figure}

The key ingredients of the proof are Weyl's estimate for the Fredholm determinant and 
Carleman's theorem from complex analysis.  If the operator $T$ is trace class on some 
Hilbert space, then Weyl's estimate (see e.g.~\cite{GK:1969}) is
\begin{equation}\label{weyl.est}
|\det (I+T)| \le \prod_{k=1}^\infty (1 + \mu_k(T)).
\end{equation}
To apply this estimate we typically break the product at some value $k=M$ to obtain
\begin{equation}\label{weyl.estM}
|\det(I+T)| \le \norm{T}^{M} \exp\biggl( \sum_{k = M+1}^\infty \mu_k(T)\biggr).
\end{equation}

The Weyl estimate allows us to apply the estimates from \S\ref{growth.sec}
to control the renormalized scattering
determinant $\vartheta_g(s)$ defined in (\ref{vartheta.def}).
\begin{lemma}\label{vt.hlf.lemma}
For $s$ in the half-plane $\re s \le 0$, we have
$$
|\vartheta_g(s)| \le e^{C\brak{s}^{n+1}},
$$
provided $d(s,-\bbN_0)>\eta$.
\end{lemma}
\begin{proof}
Define the smoothing operator on $L^2(\bX, dh)$,
$$
T(s) :=  S_g(n-s)A(s) - I,
$$
so that $\vartheta_g(s) = \det ( I+T(s))$.  By (\ref{sa.gbrb}) we can write
$$
T(s) = \Bigl(G_N(n-s)\tsp + B_N(n-s)\tsp R_g(n-s)\Bigr) B_N(s),
$$
for $|\re s-\nh| < N$.  For $\re s \ge n$, the standard
resolvent estimate gives $\norm{R_g(s)} \le 1/\inf(\sigma(\Delta_g))$.  
Thus $\norm{\rho^N R_g(n-s) \rho^N} = O(1)$ for $\re s\le 0$.
Using this along with  Lemmas~\ref{B.est} and \ref{G.est}, we deduce that
$$
\bnorm{\Bigl(G_N(n-s)\tsp + B_N(n-s)\tsp R_g(n-s)\Bigr) \rho^N} \le e^{CN},
$$
for $|s| < N/C$, $\re s \le 0$, and $d(s,-\bbN_0)>\eta$.

This allows us to deduce from Lemma~\ref{B.est} that
$$
\norm{T(s)} \le e^{CN},
$$
and 
$$
\mu_k(T(s)) \le e^{CN-ck^{1/n}}
$$
for $|s| < N/C$, $\re s \le 0$, and $d(s,-\bbN_0)>\eta$.
The Weyl estimate (\ref{weyl.estM}), broken at $M = (CN/c)^n$, then gives
$$
|\vartheta_g(s)| \le e^{CN^{n+1}},
$$
for $|s| < N/C$, $\re s \le 0$, and $d(s,-\bbN_0)>\eta$.  Since $|\vartheta_g(s)|$ is
independent of $N$, we can replace $N$ by $\brak{s}$ in the exponent by adjusting the
constant.
\end{proof}

Let us define a renormalized version of $A(s)$ analogous to (\ref{tS.def}),
$$
\tA(s) := \frac{\Gamma(s-\nh)}{\Gamma(\nh - s)}
\Lambda^{n/2-s} A(s) \Lambda^{n/2-s}.
$$ 
It follows from 
(\ref{S.AEB}) that $\tA(s)$ is also a meromorphic family of Fredholm operators
with poles of finite rank.  Since the factors $\Lambda^{n/2-s}$ are holomorphically
invertible, we can write
$$
\vartheta_g(s) = \det \tS_g(n-s) \tA(s)
$$

To compute the divisor of $\vartheta_g(s)$, we can apply 
\cite[Thm~5.2]{GS:1971} to obtain 
\[
\begin{split}
\res_\zeta \frac{\vartheta_g'}{\vartheta_g}(s) 
& =  \tr \res_\zeta \Bigl[(\tS_g(n-s) \tA(s))' \> (\tS_g(n-s) \tA(s))^{-1} \Bigr] \\
& =  - \tr \res_\zeta \bigl[\tS_g'(n-s) \tS_g(n-s)^{-1} \bigr] 
+ \tr  \res_\zeta \bigl[ \tA'(s) \tA(s)^{-1} \bigr].
\end{split}
\]
Note that the residues are finite-rank operators, so the traces here are well-defined.
To derive the second line we commuted the operators inside the second trace;  this is 
justified by a simple argument from \cite[\S4.1]{GS:1971}.
Since $\tS_g(n-s) = \tS_g(s)^{-1}$, the first term is a scattering multiplicity, 
\begin{equation}\label{res.vtheta}
\res_\zeta \frac{\vartheta_g'}{\vartheta_g}(s) = - \nu_g(\zeta) + \tr  \res_\zeta \bigl[ \tA'(s) \tA(s)^{-1} \bigr].
\end{equation}

The poles of $\tA(s)$ come from the model scattering matrix $\tS_0(s)$ on $\bbH^{n+1}$.
The operator $\tA(s)$ is the sum of a finite number of locally defined operators of the form
$$
\Lambda^{n/2-s}
\varphi_1^j \gamma_j^s\> \iota_j^*\!\left(\frac{\Gamma(s-\nh)}{\Gamma(\nh - s)} S_0(s)\right) 
\gamma_j^s \varphi^j \Lambda^{n/2-s}.
$$
By (\ref{S0.formula}), this expression has poles of order $h_n(k)$ at
$s = -k$ for $k \in \bbN_0$ (and corresponding zeroes at $s = n+k$).  We conclude 
that the only poles of $\tA(s)$ occur at $s=-k$ for $k\in \bbN_0$, with multiplicities bounded by $Ck^n$.  

\bigbreak
\begin{proof}[Proof of Proposition~\ref{left.upbound}]
For sufficiently large $p \in \bbN$, the function $h(s) := g_n(s)^p \vartheta_g(s)$ will be analytic
in the half plane $\re s \le \nh$, where $g_n(s)$ was the function introduced in (\ref{gn.def}).  
The resonances in $\Rsc_g \cap \{\re s \le 0\}$
are included among the zeros of $h(s)$, with multiplicities.  Moreover, by Lemma~\ref{vt.hlf.lemma}
and (\ref{gs.est}) we have the growth estimate
\begin{equation}\label{hs.bnd}
|h(s)| \le e^{C\brak{s}^{n+1}},
\end{equation}
for $\re s \le 0$.  (Since $h(s)$ is analytic, we can drop the restriction $d(s,-\bbN_0)>\eta$
by the maximum principle.)

Choose $\delta>0$ so that $h(-\delta) \ne 0$, and for $\vep >0$ let
$$
n(r) := \#\{\zeta\in\Rsc:\>|\zeta+\delta|\le r,\> \arg(\zeta+\delta) \in
[\tfrac{\pi}2 + \vep,\tfrac{3\pi}2 - \vep] \}
$$
By Carleman's theorem \cite[Thm.~1.2.2]{Boas}, for some $c>0$
\[
\begin{split}
c n(r/2) & \le \frac{r}2 \im h'(-\delta) + 
\frac{1}{\pi} \int_{\pi/2}^{3\pi/2} \log \left| \frac{h(-\delta + re^{i\theta})}{h(-\delta)} 
\right | \sin\theta\>d\theta \\
&\quad + \frac{r}{2\pi} \int_0^r \left( \frac{1}{y^2} - \frac{1}{r^2}\right) 
\log \left| \frac{h(-\delta+ iy) h(-\delta-iy)}{h(-\delta)^2} \right| dy.
\end{split}
\]
The right-hand side is $O(r^{n+1})$ by (\ref{hs.bnd}), and this gives the stated estimate.
\end{proof}

\section{Global determinant estimates}\label{det.est.sec}

We now turn to the analysis of the properties of the scattering determinant $\vartheta_g(s)$ 
as a meromorphic function on all of $\bbC$.
The main result of this section is the following:
\begin{proposition}\label{detSA.prop}
The function $\vartheta_g(s)$
is a ratio of entire functions of order at most $n^2+3n+2$.
\end{proposition}

This follows almost immediately from the following growth estimate:
\begin{lemma}\label{det.growth}
For each $N\in \bbN$ there exists $a_N$, contained in some fixed interval $[a,b]$, such that
$$
| \vartheta_g(s)| \le \exp [CN^{(n+1)(n+2)+\vep}]\quad\text{for }|s-n| = a_N N,
$$
\end{lemma}
\begin{proof}
As in the proof of Lemma~\ref{vt.hlf.lemma} we set
$T(s) :=  S_g(n-s)A(s) - I$.  But this time we will base the estimates on
(\ref{sa.fkb}), which for $|\re s - \nh| < N$ gives
$$
T(s) = F(n-s) (I-K_N(n-s))^{-1} B_N(s).
$$
Applying Lemmas \ref{IK.est}, \ref{F.est}, and \ref{B.est} to this expression gives
$$
\norm{T(s)} \le e^{CN^{n+2+\vep}}
$$
and
$$
\mu_k(T(s)) \le e^{CN^{n+2+\vep} - ck^{1/n}}
$$
for $|s-n| = a_N N$.  

Using the Weyl estimate (\ref{weyl.estM}), broken at $M = (CN^{n+2+\vep}/c)^{n}$,
we deduce immediately that 
$$
\det(I+|T(s)|) \le \exp [CN^{n(n+2+\vep)}N^{n+2+\vep}],
$$
for $|s-n| = a_N N$.
\end{proof}

\bigbreak
\begin{proof}[Proof of Proposition \ref{detSA.prop}]
By (\ref{res.vtheta}) and the observation that the multiplicity of the pole of 
$\tA(s)$ at $s=-k$ is bounded by $Ck^n$, we conclude that 
$$
\vartheta_g(s) = \frac{h_1(s)}{h_2(s)},
$$
where $h_1(s)$ and $h_2(s)$ are entire and $h_2(s)$ has order $n+1$.
Lemma~\ref{det.growth} now shows that $h_1(s)$ has order at most
$n^2 + 3n+2$.
\end{proof}

\section{Factorization of the scattering determinant}\label{fact.sec}

Suppose that $(X,g)$ and $(X_0, g_0)$ are two conformally compact manifolds,
hyperbolic near infinity and isometric to each other outside of some compact sets 
$K \subset X$ and $K_0 \subset X_0$.   Then we can assume that the parametrix 
construction from \S\ref{param.sec} is performed using identical constructions on 
$(X-K,g) \cong (X_0-K_0, g_0)$.  In particular, we can use the same auxiliary operator $A(s)$ in the
definitions of $\vartheta_g(s)$ and $\vartheta_{g_0}(s)$.  The formula (\ref{S.AEB})
shows that $S_g(s)$ and $S_{g_0}(s)$ differ by a smoothing operator, hence $S_g(s) S_{g_0}(s)^{-1}$
is determinant class and (\ref{detk.At}) implies
\begin{equation}\label{dss.tt}
\det S_g(s) S_{g_0}(s)^{-1} = \frac{\vartheta_{g_0}(s)}{\vartheta_g(s)}.
\end{equation}
We will call this the \emph{relative scattering determinant} for the pair $(g,g_0)$.
Proposition \ref{detSA.prop} shows that the right-hand side is a ratio of entire
functions of bounded order.  In this section we will refine this result into a Hadamard-type
factorization. 

Let $\Upsilon_g(s)$ be the meromorphic function defined by
\begin{equation}\label{Ups.def}
\Upsilon_g(s) := (2s-n) \otr [R_g(s) - R_g(n-s)],
\end{equation}
for $s \notin \bbZ/2$.   This expression is evaluated away from $\bbZ/2$ because of
anomalies that occur in the $0$-trace when the spaces $\rho^s \cinf(\barX)$ and
$\rho^{n-s}\cinf(\barX)$ intersect.  
There is a very important connection between $\Upsilon_g(s)$ and the relative scattering
determinant:
\begin{lemma}\label{dsr.ups.lemma}
With $(X,g)$ a compactly supported perturbation of $(X_0,g_0)$, in the sense described above,
for $s\notin \bbZ/2$ we have the meromorphic identity
\begin{equation}\label{dsr.RR}
- \del_s \log \det S_g(s) S_{g_0}(s)^{-1} = \Upsilon_g(s) - \Upsilon_{g_0}(s).
\end{equation}
\end{lemma} 
For the proof, we note the calculations of Patterson-Perry \cite[\S6]{PP:2001} for the
hyperbolic case apply
also when $(X, g)$ is conformally compact and hyperbolic near infinity, since they rely only
on a covering of $\bX$ by model neighborhoods of exactly the type we introduced in \S2. 
The formula (\ref{dsr.RR}) follows immediately from \cite[Lemma~6.7]{PP:2001} and
the decomposition of $S_g(s)^{-1} S_g'(s)$ introduced in the proof of \cite[Prop.~5.3]{PP:2001}.
Note also that a two-dimensional version of this result appears in the proof of
\cite[Prop.~4.5]{GZ:1997}.   And in even dimensions the result can be deduced directly 
from Guillarmou \cite[Thm.~1.2]{Gui:2005b}.

Define the Hadamard product over the scattering resonance set,
\begin{equation}\label{Pg.def}
P_g(s) := \prod_{\zeta\in \Rsc_g} \Bigl(\frac{s}{\zeta}, n+1\Bigr),
\end{equation}
using the elementary factors $E(s,p)$ defined in (\ref{elemfact}).
This converges by (\ref{Nsc.bound}) to an entire function of order $n+1$. 

\begin{proposition}\label{srdet.prop}
With $(X, g)$ a compactly supported perturbation of $(X_0, g_0)$, we have
\begin{equation}\label{detSr.factor}
\det S_g(s) S_{g_0}(s)^{-1}  = e^{q(s)} \frac{P_g(n-s)}{P_g(s)} \frac{P_{g_0}(s)}{P_{g_0}(n-s)},
\end{equation}
where $q(s)$ is a polynomial.  For a non-topological perturbation (i.e.~$X_0\cong X$) 
the degree of $q(s)$ is at most $n+1$.
\end{proposition}
\begin{proof}
Since the same $A(s)$ can be used for both $g$ and $g_0$, 
we see immediately from (\ref{res.vtheta}) that
\begin{equation}\label{res.ldss}
\res_\zeta \Bigl[\del_s \log \det S_g(s) S_{g_0}(s)^{-1} \Bigr] 
=  - \nu_g(\zeta) + \nu_{g_0}(\zeta).
\end{equation}
Then from (\ref{nu.mumu}) we obtain the right-hand side of (\ref{detSr.factor}) with $q(s)$ an entire function.
The fact that $q(s)$ a polynomial follows immediately from (\ref{dss.tt}) and
Proposition \ref{detSA.prop}. 

All that remains is
to improve the estimate on the degree in the case of a metric perturbation.   
For this purpose we will introduce the zeta-regularized relative determinant and adapt some 
arguments from Borthwick-Judge-Perry \cite{BJP:2003}.  Assume that $g$ and $g_0$
are metrics on $X$ that agree outside a compact set.
Let $L^2(X)$ denote the space of square-integrable half-densities, with $\hD_g$
and $\hD_{g_0}$ the Laplacians on $L^2(X)$ associated to the respective metrics.

By the joint parametrix construction we can see that the
operator $\hR_g(s)^m - \hR_{g_0}(s)^m$ is trace class on $L^2(X)$ for $\re s >n$
with $m=(n+3)/2$.  For $\re w\ge m$ and $\re s >n$ we introduce the relative
zeta function
$$
\zeta(w,s) := \tr \bigl[ \hR_g(s)^w - \hR_{g_0}(s)^w \bigr],
$$
which could also be written in terms of heat operators,
\begin{equation}\label{zeta.heat}
\zeta(w,s) = \frac{1}{\Gamma(w)} \int_0^\infty t^w e^{ts(n-s)} \tr \bigl[ e^{-t\hD_{g}} 
- e^{-t\hD_{g_0}} \bigr] \frac{dt}{t}.
\end{equation}
The heat expansions as $t \to 0$ can be used to show that $\zeta(w,s)$ extends meromorphically
to $\re w > -1$, with a single simple pole at $w=1$.   Hence the relative determinant,
$$
\Drel(s) := \exp \bigl[-\del_w\zeta(w,s)|_{w=0}\bigr],
$$
is well-defined for $\re s > n$.

The Birman-Krein theory of the spectral shift (see e.g.~\cite[Ch.~8]{Yafaev})
gives us a shift function $\xi(\lambda)$ for $\lambda \in [\mu_0, \infty)$, where
$\mu_0 := \inf (\sigma(\Delta_g)\cup \sigma(\Delta_{g_0}))$, such that
$\lambda^{-m-1} \xi(\lambda)$ is integrable and
\begin{equation}\label{BK.zeta}
\zeta(w,s) = -w \int_{\mu_0}^\infty (\lambda - s(n-s))^{-w-1} \xi(\lambda)\>d\lambda,
\end{equation}
for $\re w \ge m$.  This leads to the identity
\begin{equation}\label{del.Dr}
\bigl[(2s-n)^{-1} \del_s \bigr]^m \log \Drel(s) = (-1)^{m-1} (m-1)!\>  
\tr \bigl[ \hR_g(s)^m - \hR_{g_0}(s)^m \bigr],
\end{equation}
valid for $\re s \ge n$.  By introducing the 0-trace on the right, which equals
the trace for a trace class operator, we can extend this to a meromorphic identity for $s \in\bbC$.
In particular, since the resolvents are non-singular for $\re s \ge \nh$, $s \notin [\nh, n]$,
we deduce that $\Drel(s)$ extends to an analytic function of $s$ 
for $\re s \ge \nh$, $s \notin [\nh, n]$.

From Lemma~\ref{dsr.ups.lemma} and this meromorphic extension of (\ref{del.Dr}) we deduce that
$$
\bigl[(2s-n)^{-1} \del_s \bigr]^m \log \det S_g(s) S_{g_0}(s)^{-1}
=  \bigl[(2s-n)^{-1} \del_s \bigr]^m \log \frac{\Drel(n-s)}{\Drel(s)},
$$
for $\re s = \tfrac{n}2$, $s \ne \nh$.  This shows that $\Drel(s)$ has a meromorphic continuation
to all of $\bbC$ such that
\begin{equation}\label{SS.DD}
 \det S_g(s) S_{g_0}(s)^{-1}
=  e^{q_1(s)} \frac{\Drel(n-s)}{\Drel(s)},
\end{equation}
with $q_1(s)$ a polynomial of degree at most $m-1 = (n+1)/2$.

By analyzing the behavior of (\ref{del.Dr}) in the vicinity of a resonance, we can show that
the divisor of $\Drel(s)$ for $\re s \ge \nh$ coincides with that of $P_g(s)/P_{g_0}(s)$.
The proof is almost identical to that of \cite[Lemma~5.3]{BJP:2003}, except that we must
use the $m$-th power of the resolvent.  We omit the details.

Since (\ref{detSr.factor}) has already been proven with $q(s)$ polynomial, the formula
(\ref{SS.DD}) together with knowledge of the divisor of $\Drel(s)$ for $\re \ge \nh$ yields
\begin{equation}\label{Dr.PP}
\Drel(s) = e^{q_2(s)} \frac{P_g(s)}{P_{g_0}(s)},
\end{equation}
for some entire function $q_2(s)$ such that that $q_2(s) - q_2(n-s)$ is polynomial.
A separate estimate is required to prove that $q_2(s)$ is itself polynomial.  
From (\ref{BK.zeta}) we can derive 
$$
\Bigl| \tr \bigl[ \hR_g(s)^m - \hR_{g_0}(s)^m \bigr] \Bigr| \le C,
$$
for $\re s \ge \tfrac{n}2 + \vep$ and $|s| > n$.  
By (\ref{del.Dr}) gives a polynomial bound on $q_2(s)$ in this range.
Since $q_2(s) - q_2(n-s)$ is polynomial, we also get a polynomial bound 
for $\re s \le \tfrac{n}2 - \vep$ and $|s-n| > n$.

The final step in the estimate of $q_2(s)$ is a bound of the form
\begin{equation}\label{exp.strip}
\Bigl| \tr \bigl[ \hR_g(s)^q - \hR_{g_0}(s)^q \bigr] \Bigr| \le e^{C\brak{s}^M},
\end{equation}
for some $q \ge m$, in the strip $|\re s - \tfrac{n}2| \le \vep$, away from union of the resonance sets.
We can produce formulas for the kernels of $R_g(s)^q$ and $R_{g_0}(s)^q$ by
applying $[(2s-1)^{-1} \del_s]^{q-1}$ to the parametrix expression $R_g(s) = M_N(s)(I-K_N(s))^{-1}$.
(Taking $N=1$ will suffice here.)  To obtain the bound, we can exploit the fact that
$(2s-1)^{-1} \del_s K_N(s)$ is independent of $s$ and can be decomposed as a compactly supported operator
of order $-2$ plus a smoothing operator.  Explicit estimates for the kernel of this smoothing
term (and its derivatives) follow from \cite[Lemmas~4.1 and 4.2]{GZ:1995b}.   We can 
thereby estimate the Schatten class norms of these terms.    
Operator norm estimates of $\rho^N M_N(s) \rho^N$ follows from the same lemmas.  That 
leaves factors of $\rho^{-N} (I-K_N(s))^{-1}\rho^{N}$, whose operator norm is estimated 
away from the resonances in Lemma~\ref{IK.est}.

From (\ref{exp.strip}) we obtain an exponential estimate on $q_2(s)$ in the strip 
$|\re s - \tfrac{n}2| \le \vep$ to complement the polynomial bounds for $|\re s - \tfrac{n}2| \ge \vep$.
The Phragmen-Lindel\"of theorem implies that $q_2(s)$ is a polynomial.
Returning now to (\ref{Dr.PP}), we consider the logarithm of this equation,
$$
q_2(s) = \log \Drel(s) + \log P_{g_0}(s) - \log P_{g}(s),
$$
as $\re s \to \infty$.  The logarithms of the Hadamard produces are bounded
by $|s|^{n+1+\vep}$.  And by applying the heat kernel expansion in (\ref{zeta.heat}),
we can derive the asymptotic
$$
\log \Drel(s) \sim a_0 (s(n-s))^{(n+1)/2} \log [s(s-n)],
$$
as $\re s \to \infty$.  Since $q_2(s)$ is already known to be polynomial, these asymptotics
imply that $q_2(s)$ has degree at most $n+1$.  

Now if we substitute (\ref{Dr.PP}) into (\ref{SS.DD}), we can derive (\ref{detSr.factor})
with $q(s)$ a polynomial of degree at most $n+1$.
\end{proof}

\bigbreak
For $\dim X$ even ($n$ odd), Guillarmou \cite{Gui:2005b} introduced a regularized
determinant of $\tS_g(s)$ based on the Kontsevich-Vishik trace, which we will denote
by $\detkv \tS_g(s)$.   This definition is intrinsic, in contrast to $\vartheta_g(s)$ which
includes the ad hoc contribution from $A(s)$.  The tradeoff is that it seems quite
difficult to obtain direct growth estimates for $\detkv \tS_g(s)$, whereas $\vartheta_g(s)$
was defined precisely so it could be estimated easily.  

These two regularizations are related.
Since $S_g(n-s)A(s)$ is determinant class, by \cite[Prop.~27]{KV:1994} we have
\begin{equation}\label{detk.At}
\detkv \tS_g(s) = \frac{\detkv A(s)}{\vartheta_g(s)},
\end{equation}
Despite the explicit formula we have for the kernel $A(s)$, it appears quite difficult to analyze 
$\detkv A(s)$ directly.  For our application (the Poisson formula) we would need polynomial growth estimates
on the renormalized trace $\tr_{\rm KV} A(s)^{-1} A'(s)$.  The singular-value techniques
which were crucial in the estimation of $\vartheta_g(s)$ are not available for this purpose.

We can, however, use the relative scattering determinant to get some information about
$\detkv \tS_g(s)$.  In context of Proposition~\ref{srdet.prop}, the combination of
(\ref{dss.tt}) and (\ref{detk.At}) implies that
$$
\frac{\detkv \tS_g(s)}{\detkv \tS_{g_0}(s)} = \det S_g(s) S_{g_0}(s)^{-1}.
$$
If the metric $g_0$ is hyperbolic, then 
the functional equation for the Selberg zeta function given by Guillarmou \cite[Thm.~1.3]{Gui:2005b}
shows that $\detkv S_{g_0}(s)$ is a ratio of entire functions of order $n+1$, with divisor equal
to that of $P_{g_0}(n-s)/P_{g_0}(s)$.  Thus
Proposition~\ref{srdet.prop} gives the following:
\begin{proposition}\label{sdet.prop}
Suppose that $\dim X$ is even and $(X,g)$ is a compactly 
supported perturbation of a conformally compact hyperbolic manifold $(X_0, g_0)$.
Then the KV-determinant of the scattering matrix admits a Hadamard factorization:
\begin{equation}\label{detkS.factor}
\detkv \tS_g(s) = e^{q(s)} \frac{P_g(n-s)}{P_g(s)},
\end{equation}
where $q(s)$ is a polynomial.  For a non-topological perturbation ($X\cong X_0$), 
the degree of $q(s)$ is at most $n+1$.
\end{proposition}

\section{Poisson formula}\label{poisson.sec}

In this section, we will continue to assume that $(X,g)$ and $(X_0,g_0)$ are isometric
outside of some compact sets.  We will assume in addition that the background manifold
$(X_0,g_0)$ is conformally compact hyperbolic.

The wave $0$-trace is defined as a distribution on $\bbR$ by
$$
\Theta(t) := \otr \left[ \cos \left(t \sqrt{\smash[b]{\Delta_g - n^2/4}}\,\right) \right].
$$
This can be separated into contributions from the discrete and continuous
spectrum, $\Theta(t) = \Theta_{\rm d}(t) + \Theta_{\rm c}(t)$.  The discrete part is
given by an actual trace,
\begin{equation}\label{t.disc}
\Theta_{\rm d}(t) = \frac12 \sum_{\res \zeta > \frac{n}2} \Bigl(e^{(\zeta-n/2)t} - e^{(n/2-\zeta)t}  \Bigr).
\end{equation}
On the other hand, the functional calculus gives a formula for the continuous part:
for $\phi \in \cinf_0(\bbR)$,
\begin{equation}\label{funct.calc}
\int_{-\infty}^\infty \phi(t) \Theta_{\rm c}(t)\>dt = \frac{1}{4\pi} \int_{-\infty}^\infty (2i\xi)
\otr \Bigl[R_g(\tfrac{n}2 + i\xi) - R_g(\tfrac{n}2 - i\xi)\Bigr] \hat\phi(\xi)\>d\xi.
\end{equation}
The integrand on the right-hand side is equal to $\Upsilon_g(\tfrac{n}2+i\xi)$
for $\xi \ne 0$ by definition.  

In the hyperbolic case, the formula of Patterson-Perry
\cite[eq.~(6.7)]{PP:2001} expresses $\Upsilon_{g_0}(s)$ in terms of the logarithmic 
derivative of the Selberg zeta function:
\begin{equation}\label{ups.zeta}
\Upsilon_{g_0}(s) = \frac{Z_{g_0}'}{Z_{g_0}}(s) + \frac{Z_{g_0}'}{Z_{g_0}}(n-s)
+  \kappa_0(s),
\end{equation}
where $q_1(s)$ is a polynomial of degree at most $n+1$, and $\kappa_0(s)$ is a topological
term given by
$$
\kappa_0(s) := \pi^{-n/2} \frac{\Gamma(\nh) \Gamma(s) \Gamma(n-s)}{ \Gamma(n) 
\Gamma(\nh-s)\Gamma(s-\nh)} \ovol(X_0,g_0).
$$
This shows in particular that $\Upsilon_{g_0}(s)$ extends
meromorphically from $\bbC- \bbZ/2$ to all of $\bbC$.

\begin{lemma}\label{upsing.lemma}
Assume that $(X,g)$ is conformally compact and hyperbolic near infinity.
For $\psi \in \cinf_0(\bbR)$ we have
\[
\begin{split}
& \frac{1}{4\pi} \int_{-\infty}^\infty (2i\xi)
\otr \Bigl[R_g(\tfrac{n}2 + i\xi) - R_g(\tfrac{n}2 - i\xi)\Bigr] \psi(\xi)\>d\xi \\
&\qquad=  \frac{1}{4\pi} \int_{-\infty}^\infty \Upsilon_g(\tfrac{n}2 + i\xi) \psi(\xi)\>d\xi + \frac12 m_{n/2}\psi(0).
\end{split}
\]
\end{lemma}
\begin{proof}
Away from $\xi = 0$, this formula reduces to the definition of $\Upsilon_g(s)$, so the point of
the proof is to compute the anomaly in the $0$-trace at $\xi =0$.
This anomaly occurs exactly as in the proof of \cite[Prop.~4.5]{GZ:1997}.
(See also \cite[Lemma.~11.5]{Borthwick} for an expository treatment.)

In the present context, the relevant computations are done by Patterson-Perry in \cite[\S6.1]{PP:2001}.
They use an integration by parts inspired by the Maass-Selberg relation to reduce the integral, 
$$
I_\vep(s) := (2s-n) \int_{\rho \ge \vep} \Bigl[ R_g(s;z,z') - R_g(n-s;z,z') \Bigr]\Big|_{z=z'}
\>dg(z),
$$
to a sum of three terms, namely, 
\[
\begin{split}
I_\vep^1(s) & = \int_{\rho = \vep} \Bigl[ R_g(s;z,z') - R_g(n-s;z,z') \Bigr]\Big|_{z=z'}
\>d\sigma_\vep(z), \\
I_\vep^2(s) & = (2s-n)^{-1} \int_\bX \int_{\rho = \vep} \Biggl[ E_g(n-s;z,x') \>\del_s(-\rho \del_\rho + s) E_g(s;z,x') \\
&\hskip1in + (-\rho \del_\rho+n-s) E_g(n-s;z,x')\>  \del_s E_g(s;z,x') \Biggr]
\>d\sigma_\vep(z)\> dh(x'), \\
I_\vep^3(s) & = - \int_\bX \int_{\rho = \vep} E_g(n-s;z,x)\> \del_s E_g(s;z,x')\>d\sigma_\vep(z)\> dh(x'), 
\end{split}
\]
where $\sigma_\vep$ is the metric induced on $\{\rho=\vep\}$ by $g$.  The anomaly
we are interested is caused by factors of the form $\vep^{\pm(n-2s)}$ occuring
in the asymptotic expansion of $I_\vep(s)$ as $\vep \to 0$.  From the analysis in 
\cite[Lemmas~6.5 and 6.7]{PP:2001}, we can see that such terms
do not occur in either $I_\vep^2(s)$ or $I_\vep^3(s)$.
However, the expansion of $I_\vep^1(s)$ does contain terms of this form.
Patterson-Perry showed that
$$ 
\FPe I_\vep^1(s) = 0, \qquad \text{for }\re s = \nh, s\ne \nh.
$$ 
Their argument does not extend to $s = \nh$ if there is a resonance there.

Suppose that a resonance of multiplicity $m_{n/2}$ occurs at $s=\tfrac{n}2$.
We first argue that such a resonance must be simple.  By self-adjointness
of the Laplacian, for $\varphi \in \cinf_0(X)$,
$$
\im \int_X  \overline{\varphi}\> \Ds \varphi\>dg
= \im (s^2-ns) \>\norm{\varphi}^2.
$$
This leads directly to a resolvent estimate
\begin{equation}\label{res.est}
\norm{R_g(s)} \le \frac{1}{| \im (s^2-ns)|},
\end{equation}
for $\re s >\nh$, which shows that the order of a pole at $s=\nh$ is at most two.
Using the relation
$$
(\Delta_g - n^2/4) R_g(s) = I - (s - \nh)^2 R_g(s),
$$
along with (\ref{res.est}), we can deduce that the range of the order two component
of the singular part consists of $L^2$ eigenfunctions with eigenvalue $n^2/4$.  By
Mazzeo's unique continuation result \cite{Ma91a}, 
an asymptotically hyperbolic manifold has no eigenvalue at $n^2/4$. 
Hence the order of the pole of $R_g(s)$ at $s = \nh$ is one.

Because the pole is simple, near $s = \nh$ the resolvent will have the structure
$$
R_g(s) = \sum_{j=1}^{m_{n/2}} \frac{\phi_j(s)\otimes \phi_j(s)}{2s-n} + \text{holomorphic},
$$
for some families of functions $\phi_j(s) \in \rho^s \cinf(\barX)$ such that 
$\{\phi_j(\tfrac{n}2)\}$ are independent.  When we substitute this expression into $I_\vep^1(s)$, 
the holomorphic part does not contribute to the finite part as $\vep \to 0$, but from 
the singular part we obtain
$$
\FPe  I_\vep^1(\nh + i\xi) 
= \tr \Biggl( \sum_{j=1}^{m_{n/2}} \phi_j^\sharp(\nh)\otimes \phi_j^\sharp(\nh)\Biggr)\> 
\lim_{\vep\to 0} \frac{\vep^{-2i\xi} - \vep^{2i\xi}}{2i\xi},
$$
where $\phi_j^\sharp(\nh) := (\rho^{-n/2} \phi_j(\nh))|_{\bX}$ and the limit is interpreted 
as a distribution on $\bbR$.
An elementary distributional calculation gives
$$
\lim_{\vep\to 0} \frac{\vep^{-2i\xi} - \vep^{2i\xi}}{2i\xi}
= \pi \delta(\xi).
$$

Analyzing the scattering
matrix near $s = \nh$ as in \cite[Lemma~4.16]{PP:2001} shows that
$$
S_g(\nh) = -I + 2P,
$$
where
$$
P := \frac12 \sum_{j=1}^{m_{n/2}} \phi_j^\sharp(\nh)\otimes \phi_j^\sharp(\nh).
$$
Since $S_g(\nh)$ is self-adjoint and $S_g(\nh)^2 = I$, we deduce that
$P$ is an orthogonal projection.  
Furthermore, $P$ has maximal rank $m_{n/2}$, because otherwise some combination of the
$\phi_j(\nh)$'s would give an $L^2$ eigenfunction.  Thus we have
$$
\tr \Biggl( \sum_{j=1}^{m_{n/2}} \phi_j^\sharp(\nh)\otimes \phi_j^\sharp(\nh)\Biggr)
= 2 \tr P = 2m_{n/2},
$$
and hence
$$
\FPe  I_\vep^1(\nh + i\xi)  =  2\pi m_{n/2}\> \delta(\xi).
$$
\end{proof}

\bigbreak
Patterson-Perry \cite[Thm.~1.9]{PP:2001} also proved the functional equation
$$
\frac{Z_{g_0}(s)}{Z_{g_0}(n-s)} = e^{p_1(s)} \left(\frac{Z_0(s)}{Z_0(n-s)}\right)^{-\chi(X_0)}
\frac{P_{g_0}(s)}{P_{g_0}(n-s)},
$$
where $p_1(s)$ is a polynomial of degree at most $n+1$ and 
$$
Z_0(s) = s \prod_{k=1}^\infty E(-\tfrac{s}{k}, n+1)^{h_n(k)},
$$
with $h_n(k)$ as defined in (\ref{nu0.def}).  Combining these formulas with (\ref{ups.zeta}),
Lemma~\ref{dsr.ups.lemma} and Proposition~\ref{srdet.prop} yields
\begin{equation}\label{ups.pp}
\Upsilon_g(s) =  \del_s \log \left[e^{p(s)} \frac{P_g(s)}{P_g(n-s)} 
\left(\frac{Z_0(s)}{Z_0(n-s)}\right)^{-\chi(X_0)} \right]
+ \kappa_0(s),
\end{equation}
where $P_g(s)$ is the Hadamard product over $\Rsc_g$ and $p(s)$ is a polynomial.  
This formula is the essential ingredient in the Poisson formula.

To obtain a formula for $\Theta_{\rm c}(t)$, we need take the Fourier
transform of $\Upsilon_g(\tfrac{n}2+i\xi)$.  This is justified provided the
latter function defines a tempered distribution for $\xi \in \bbR$.  The fact that 
$\del_\xi \log (P(\nh+i\xi)/P(\nh-i\xi))$ is tempered can be proven by
an easy adaptation of the proof given for the case $n=1$ by 
Guillop\'e-Zworski \cite[Lemma 4.7]{GZ:1997}.  
Because the other terms in (\ref{ups.pp}) clearly
satisfy polynomial bounds when restricted to $\re s = \nh$, this shows that
$\Upsilon(\nh+i\xi)$ is tempered.

\bigbreak
\begin{proof}[Proof of Theorem \ref{poisson.thm} (Poisson formula)]
Since $\Upsilon(\nh+i\xi)$ defines a tempered distribution, 
Lemma \ref{upsing.lemma} and (\ref{funct.calc}) give us the formula
\begin{equation}\label{theta.fups}
\Theta_{\rm c}(t) = \frac{1}{4\pi} \calF\bigl[\Upsilon_g(\nh+i\xi)\bigr](t) + \frac12 m_{n/2}.
\end{equation}
By (\ref{ups.pp}) we can write 
\begin{equation}\label{ups.expand}
\Upsilon_g(\nh+i\xi) = p'(\nh+i\xi) + \upsilon_1(\xi) - \chi(X_0) \upsilon_2(\xi) + \kappa_0(\nh+i\xi),
\end{equation}
where
$$
\upsilon_1(\xi) := \sum_{\zeta\in\Rsc_g}  \left(\frac{n-2\zeta}{\xi^2 + (\zeta-n/2)^2}
+ p_n(\zeta; \xi) \right)
$$
and
$$
\upsilon_2(\xi) :=  \sum_{k=0}^\infty  h_n(k)\left(\frac{n+2k}{\xi^2 + (k+n/2)^2}
+ p_n(-k; \xi) \right),
$$
with $p_n(\zeta; \xi)$ is a polynomial of degree $n$ in $\xi$.

If $\upsilon_1(\xi)$ is differentiated $n+1$ times, the polynomial terms 
drop out, so that
$$
\del_\xi^{n+1} \upsilon_1(\xi) = \sum_{\zeta\in \Rsc_g} \del_\xi^{n+1} \left(\frac{n-2\zeta}{\xi^2 + (\zeta-n/2)^2}\right).
$$
By a simple contour integration, 
$$
\int_{-\infty}^\infty e^{-i\xi t} \frac{n-2\zeta}{\xi^2 + (\zeta-n/2)^2}\>d\xi = 
\begin{cases}
- 2\pi e^{(n/2-\zeta)|t|} & \re \zeta > \tfrac{n}2, \\
2\pi e^{(\zeta-n/2)|t|} & \re \zeta < \tfrac{n}2.
\end{cases}
$$
We can thus compute
$$
t^{n+1} \widehat{\upsilon}_1(t)  = 2\pi t^{n+1} \Biggl( 
\sum_{\re\zeta<n/2} e^{(\zeta-n/2)|t|} - \sum_{\re\zeta>n/2}e^{(n/2-\zeta)|t|}  \Biggr).
$$
By the same argument
\[
\begin{split}
t^{n+1} \widehat{\upsilon}_2(t) & = 2\pi t^{n+1} \sum_{k=0}^\infty h_n(k) e^{-(k+n/2)|t|} \\
& = 2\pi t^{n+1} \frac{2 \cosh t/2}{(2\sinh |t|/2)^{n+1}}.
\end{split}
\]

Since $p'(\nh+i\xi)$ is polynomial, its Fourier transform is a distribution
supported at $0$.  The same is true of $\kappa_0(\nh+i\xi)$ is $n$ is even.

If $n$ is odd, a simple calculation with the gamma function shows that the residue of the
simple poles of $\kappa_0(s)$ at $-k$ and $n+k$ is
$$
\frac{n!!}{(-2\pi)^{\frac{n+1}2}} h_n(k) \ovol(X_0, g_0),
$$
for $k\in\bbN_0$.  By Epstein's formula from \cite[Thm.~A.1]{PP:2001},
$$
\ovol(X_0, g_0) = \frac{(-2\pi)^{\frac{n+1}2}}{n!!} \chi(X_0).
$$
Thus the residue of the poles of $\kappa_0(s)$ at $-k$ and $n+k$ is given
by $\chi(X_0) h_n(k)$.  Since $\kappa_0(s)$ has polynomial growth of degree $n$,
away from the poles, it follows that $\kappa_0(\nh+i\xi)$ is equal to 
$\chi(X_0) \upsilon_2(\xi)$ plus a polynomial of degree $n$.  Hence
$$
t^{n+1} \calF\bigl[\kappa_0(\nh+i\xi)\bigr](t) = t^{n+1}\chi(X_0) \widehat{\upsilon}_2(t)
$$
for $n$ odd, and these terms cancel each other out of the Poisson formula.
(In \cite[eq.~(2.11)]{GN:2006} this same fact was deduced from the wave $0$-trace
on $\bbH^{n+1}$.)

Returning now to (\ref{theta.fups}), we have computed that
$$
t^{m} \Theta_{\rm c}(t) = \frac{t^{m}}2 \Biggl( 
\sum_{\re\zeta\le n/2} e^{(\zeta-n/2)t} - \sum_{\re\zeta>n/2}e^{(n/2-\zeta)t}  
- A(X) \frac{2 \cosh t/2}{(2\sinh |t|/2)^{n+1}} \Biggr),
$$
for $m = \max(\deg p(s), n+1)$, 
where $A(X) = 0$ for $n$ odd and $\chi(X_0)$ for $n$ even.  In the latter case we note that
$\chi(X_0) = \tfrac12(\bX) = \chi(X)$ because the dimension of $X$ is odd.
Combining this result with (\ref{t.disc}) completes the proof.
\end{proof}

\section{Lower bounds on resonances}\label{lbr.sec}

Deducing spectral asymptotics from the small-$t$ behavior of the
wave trace is a well-established technique spectral theory.  For asymptotically hyperbolic
manifolds, Joshi-S\'a Barreto \cite{JS:2001} studied the wave group and argued that
its kernel has the same local asymptotics as in the compact case\footnote{However, 
there are several typos in \cite[Prop.~4.2 and Prop.~4.3]{JS:2001}.}, as worked out 
by H\"ormander \cite[\S3]{Hormander:1968} and Duistermaat-Guillemin 
\cite[Prop.~2.1]{DG:1975}.  In particular, we have the following:
\begin{proposition}\label{wt.sing}
Assume that $(X,g)$ is asymptotically hyperbolic (conformally compact with $|d\rho|_{\bar g}=1$
on $\bX$).
If $\psi \in \cinf_0(\bbR)$ has support 
in a sufficiently small neighborhood of $0$ and $\psi = 1$ in some smaller
neighborhood of $0$, then 
\begin{equation}\label{wave.asymp}
\int_{-\infty}^\infty e^{-it\xi} \psi(t) \Theta(t)\>dt \sim \sum_{k=0}^\infty 
a_k |\xi|^{n-2k},
\end{equation}
where 
$$
a_0 =   \frac{2^{-n} \pi^{-\frac{n-1}2}}{\Gamma(\frac{n+1}{2})} \ovol(X,g).
$$ 
\end{proposition}
For exactly hyperbolic metrics, we could deduce the local form of the
$t=0$ singularity of the wave group from the model wave operator on $\bbH^{n+1}$, 
which was given explicitly in Lax-Phillips \cite[\S5]{LP:1982}.   Thus, in the case of compactly supported
perturbations of hyperbolic metrics, one could give an alternative proof of Proposition~\ref{wt.sing} 
by using a partition of unity and finite propagation 
speed to reduce to a combination of results from the conformally compact hyperbolic and compact cases.

The arguments from Guillop\'e-Zworski \cite[\S6]{GZ:1997}, which were
adapted from Sj\"ostrand-Zworski \cite{SZ:1994}, may now be applied to give the lower bound
on scattering resonances.  Since there are slight variations between even and odd dimensions, 
we include the details.

\begin{lemma}\label{hp.zeta}
For $\phi\in \cinf_0(\bbR_+)$, and $\lambda>0$ sufficiently large
\begin{equation}\label{hp.zeta.B}
\sum_{\zeta\in\Rsc_g} \widehat\phi(i(\zeta - \nh)/\lambda) \ge c B(X,g) \lambda^{n+1},
\end{equation}
where $B(X,g) = |\ovol(X,g)|$ if $n$ is odd and $|\chi(X)|$ if $n$ is even.
\end{lemma}
\begin{proof}
For $\phi\in \cinf_0(\bbR_+)$ with $\phi(t) \ge 0$ and $\phi(1)>0$ we define the rescaled function
$$
\phi_\lambda(t) := \lambda \phi(\lambda t).
$$
For $\psi$ as in (\ref{wave.asymp}) and $\lambda$ sufficiently large, we have
\begin{equation}\label{pt.ppt}
\int_{-\infty}^\infty \phi_\lambda(t) \Theta(t)\>dt =  \int_{-\infty}^\infty
\phi_\lambda(t) \psi(t) \Theta(t)\>dt.
\end{equation}
By (\ref{wave.asymp}),
$$
\widehat{\psi \Theta}(\xi) = a_0 |\xi|^{n} + r(\xi),
$$
where the remainder $r(\xi)$ is smooth  away from $\xi=0$ and 
$O(|\xi|^{n-2})$ as $\xi \to \pm \infty$.
Using the Fourier transform on the right side of (\ref{pt.ppt}) then gives
$$
\int_{-\infty}^\infty \phi_\lambda(t) \Theta(t)\>dt =  (2\pi)^{-1}  \int_{-\infty}^\infty 
\widehat{\phi_\lambda}(\xi)\> \bigl(a_0 |\xi|^n + r(\xi)\bigr)\>d\xi.
$$
Since $\widehat{\phi_\lambda}(\xi) = \widehat\phi(\xi/\lambda)$, we have
$$
\int_{-\infty}^\infty \widehat{\phi_\lambda}(\xi)\>|\xi|^n \>d\xi
= \lambda^{n+1} \int_{-\infty}^\infty \widehat{\phi}(\xi)\> |\xi|^n \>d\xi.
$$

If $n$ is odd, then since $\phi$ is compactly supported in $\bbR_+$ we can compute
\begin{equation}\label{hpl.xn}
\int_{-\infty}^\infty \widehat{\phi}(\xi)\> |\xi|^n \>d\xi = 2n! (-1)^{\frac{n+1}{2}} \int_0^\infty \phi(t) \>t^{-n-1}\>dt,
\end{equation}
and the remainder term is easily controlled by
$$
\int_{-\infty}^\infty \widehat{\phi_\lambda}(\xi) r(\xi) \>d\xi
\le C \lambda^{n-1} \int_{-\infty}^\infty |\widehat{\phi}(\xi)| \>d\xi.
$$
Thus, for $\lambda$ sufficiently large,
\begin{equation}\label{phil.lw}
\left|\int_{-\infty}^\infty \phi_\lambda(t) \Theta(t)\>dt \right| \ge c \>|a_0|\> \lambda^{n+1}.
\end{equation}
Applying the even-dimensional case of Theorem~\ref{poisson.thm} gives
\begin{equation}\label{pp.hp}
\int_{-\infty}^\infty \phi_\lambda(t) \Theta(t)\>dt = \sum_{\zeta\in\Rsc_g}
\widehat\phi(i(\zeta - \nh)/\lambda),
\end{equation}
and this completes the proof.

If $n$ is even, the integral in (\ref{hpl.xn}) gives
$$
\int_{-\infty}^\infty \widehat{\phi}(\xi)\> \xi^n \>d\xi  = C \phi^{(n)}(0) = 0.
$$
Hence 
$$
\int_{-\infty}^\infty \phi_\lambda(t) \Theta(t)\>dt = O(\lambda^{n-1}).
$$
In this case Theorem~\ref{poisson.thm} gives
$$
\sum_{\zeta\in\Rsc_g} \widehat\phi(i(\zeta - \nh)/\lambda) = 
\chi(X) \int_0^\infty  \frac{\cosh t/2\lambda}{(2\sinh t/2\lambda)^{n+1}}
\phi(t)\>dt + O(\lambda^{n-1}),
$$
where the integral is well-defined because $\phi$ is supported away from $0$.
The result for $n$ even now follows from the easy estimate
$$
\frac{\cosh t/2}{(2\sinh t/2)^{n+1}} \ge c t^{-(n+1)},
$$
for $t$ near $0$.
\end{proof}

\begin{proof}[Proof of Theorem~\ref{lowerbound.thm}]
Take $\phi$ as in Lemma~\ref{hp.zeta}.
Since $\phi$ is compactly supported in $\bbR_+$, we have analytic estimates on its
Fourier transform,
$$
|\widehat\phi(\xi)| \le C_m (1+ |\xi|)^{-m},
$$
for any $m\in\bbN$.  Combining this with (\ref{hp.zeta.B}) gives
$$
c B(X,g)  \lambda^{n+1} \le C \sum_{\zeta\in\Rsc_g} (1+ |\zeta|/\lambda)^{-m}.
$$
Writing the right-hand side as a Stieljes integral then gives
\[
\begin{split}
c B(X,g)   \lambda^{n+1} & \le C \int_0^\infty (1+ r/\lambda)^{-m} \>d\Nsc(r) \\
& = C  \int_0^\infty (1+ r)^{-m-1} \Nsc(\lambda r)\>dr.
\end{split}
\]
We then split the integral and use the upper bound (\ref{Nsc.bound}) to obtain
\[
\begin{split}
c B(X,g)   \lambda^{n+1} & \le C \int_0^b (1+ r)^{-m-1} \Nsc(\lambda r)\>dr  +
C \lambda^{n+1} \int_b^\infty r^{n+1}(1+ r)^{-m-1} \>dr \\
& \le C \Nsc(\lambda b) + C \lambda^{n+1} b^{n-m+1}.
\end{split}
\]
Setting $m=n+2$, we complete the proof by taking $b$ sufficiently small.
\end{proof}

\section{Scattering phase asymptotics}\label{scphase.sec}

In view of the definition of the relative scattering phase used by Guillop\'e-Zworski 
\cite{GZ:1997}, it makes sense to define the absolute scattering
phase associated to an asymptotically hyperbolic metric $(X, g)$ by
$$
\sigma_g(\xi) := \frac{1}{2\pi} \int_0^\xi \Upsilon(\nh+it)\>dt.
$$
This is also equal to the generalized Krein function introduced by Guillarmou \cite{Gui:2005b} in
the even-dimensional case.  It is not so clear that this is a useful definition in odd dimensions, 
however, because of the dependence on the defining function $\rho$.
In the perturbative case we could avoid this
issue by using the relative scattering phase $\sigma_g(\xi) - \sigma_{g_0}(\xi)$,
but it would be more satisfying to find an intrinsic regularization of the scattering phase in odd dimensions.

\begin{theorem}\label{scphase.thm}
For $(X,g)$ a compactly supported perturbation of a conformally compact hyperbolic manifold,
$$
\sigma_g(\xi) = \frac{(4\pi)^{-(n+1)/2}}{\Gamma(\frac{n+3}{2})} \ovol(X,g) \>\xi^{n+1} + O(\xi^n).
$$
as $\xi \to +\infty$.
\end{theorem}

For $(X,g)$ even-dimensional and conformally compact hyperbolic, Theorem~\ref{scphase.thm}
was proven by Guillarmou \cite{Gui:2005b}.
We will only sketch the details of the proof, since the argument from Guillop\'e-Zworski \cite{GZ:1997}
applies with only minor changes.  By (\ref{theta.fups}), we have
$$
\sigma_g'(\xi) = \frac{1}{2\pi} \Upsilon_X(\nh + i\xi) = \frac{1}{\pi} \widehat\Theta_{\rm c}(-\xi) - 
m_{n/2} \delta(\xi).
$$
Thus Proposition~\ref{wt.sing} implies that
$$
(\sigma_g'*\phi)(\xi) = \frac{a_0}{\pi} |\xi|^n + O(\xi^{n-2}).
$$
for $\phi \in \mathcal{S}(\bbR)$ such that $\phi >0$,
$\hat\phi=1$ near $0$, and $\hat\phi$ has support in a sufficiently small neighborhood of $0$.
This gives
$$
(\sigma_g*\phi)(\xi) = \frac{a_0}{\pi(n+1)} |\xi|^{n+1} + O(\xi^{n-1}),
$$
as $\xi \to +\infty$.
The proof then reduces to an application of Melrose's argument \cite{Melrose:1988}
to derive $\sigma_g(\xi) - \sigma_g*\phi(\xi) = O(\xi^n)$ from (\ref{ups.expand}).

\end{document}